\documentclass[12pt]{article}

\usepackage[margin=2.5cm]{geometry}

\usepackage{amsmath, amssymb, latexsym, amsthm}

\usepackage{graphicx}
\usepackage[dvipsnames]{color}
\usepackage{tikz}

\newtheorem{theorem}{Theorem}[section]
\newtheorem{proposition}[theorem]{Proposition} 

\newtheorem{corollary}[theorem]{Corollary}
\newtheorem{lemma}[theorem]{Lemma}
\newtheorem{conjecture}[theorem]{Conjecture}

\newtheorem*{theorem*}{Theorem}
\newtheorem*{conjecture*}{Conjecture}
\newtheorem*{corollary*}{Corollary}
\newtheorem*{proposition*}{Proposition}

\newtheorem{problem}[theorem]{Problem}

\theoremstyle{definition}

\theoremstyle{remark}

\definecolor{green}{cmyk}{1.0,0.2,0.7,0.07}
\definecolor{mag}{cmyk}{0.0,0.9,0.3,0.0}

\begin{document}

\title{Interval edge-colorings of Cartesian products of graphs II}

\author{Petros A. Petrosyan, Hrant H. Khachatrian, Hovhannes G. Tananyan}

\author{
{\sl Petros A. Petrosyan}\thanks{{\it E-mail address:} 
petros\_petrosyan@ysu.am}\\ 
Department of Informatics \\ and Applied Mathematics,\\
Yerevan State University \\ 0025, Armenia
\and
{\sl Hrant H. Khachatrian}\thanks{{\it E-mail address:} 
hrant@yerevann.com}\\ 
YerevaNN, \\ 
Department of Informatics \\ and Applied Mathematics,\\
Yerevan State University \\ 0025, Armenia
\and
{\sl Hovhannes G. Tananyan}\thanks{{\it E-mail address:} 
HTananyan@yahoo.com} \\
Department of Applied Mathematics\\
and Informatics,\\
Russian-Armenian University\\
0051, Armenia
}
\maketitle

\begin{abstract}
An \emph{interval $t$-coloring} of a graph $G$ is a proper edge-coloring with colors $1,\dots,t$ such that the colors on the edges incident to every vertex of $G$ are colored by consecutive colors. A graph $G$ is called \emph{interval colorable} if it has an interval $t$-coloring for some positive integer $t$. Let $\mathfrak{N}$ be the set of all interval colorable graphs. For a graph $G\in \mathfrak{N}$, we denote by $w(G)$ and $W(G)$ the minimum and maximum number of colors in an interval coloring of a graph $G$, respectively. In this paper we present some new sharp bounds on $W(G\square H)$ for graphs $G$ and $H$ satisfying various conditions. In particular, we show that if $G,H\in \mathfrak{N}$ and $H$ is an $r$-regular graph, then $W(G\square H)\geq W(G)+W(H)+r$. We also derive a new upper bound on $W(G)$ for interval colorable connected graphs with additional distance conditions. Based on these bounds, we improve known lower and upper bounds on $W(C_{2n_{1}}\square C_{2n_{2}}\square\cdots \square C_{2n_{k}})$ for $k$-dimensional tori $C_{2n_{1}}\square C_{2n_{2}}\square\cdots \square C_{2n_{k}}$ and on $W(K_{2n_{1}}\square K_{2n_{2}}\square\cdots \square K_{2n_{k}})$ for Hamming graphs $K_{2n_{1}}\square K_{2n_{2}}\square\cdots \square K_{2n_{k}}$, and these new bounds coincide with each other for hypercubes. Finally, we give several results on interval colorings of Fibonacci cubes $\Gamma_{n}$.
\end{abstract}



\section{Introduction}

We use \cite{West} for terminology and notation not defined here. We consider graphs that are finite, undirected, and have no loops or multiple edges. Let $V(G)$ and $E(G)$ denote the sets of vertices and edges of a graph $G$, respectively. The degree of a vertex $v\in V(G)$ is denoted by $d_{G}(v)$, the maximum degree of $G$ by $\Delta(G)$ and the chromatic index of $G$ by $\chi^{\prime}(G)$. A proper edge-coloring of a graph $G$ is a mapping $\alpha: E(G)\rightarrow \mathbb{N}$ such that $\alpha(e)\not=\alpha(e')$ for every pair of adjacent edges $e$ and $e'$ in $G$. 

The classical graph coloring problem is the problem of assigning positive integers (colors) to the vertices or edges of a graph so that no two adjacent vertices/edges receive the same color. However, many problems of scheduling theory can be reduced not just to classical graph coloring problems, but to ones of existence and construction of proper vertex or edge colorings with additional constraints. For example, list colorings of graphs \cite{Vizing,EdrosRubinTaylor,JensenToft,Marx} are proper vertex colorings for which each vertex receives a color from its list of available colors, and this type of vertex coloring is used to model a scheduling, where each job can be processed in certain time slots or if each job can be processed by certain machines. Another interesting example is interval edge-colorings of graphs \cite{AsrKam,AsrKamJCTB}. Interval edge-colorings are proper edge-colorings such that the colors on the edges incident to every vertex are colored by consecutive colors, and these edge-colorings are used to model timetables with a compactness requirement. Formally, an \emph{interval $t$-coloring} of a graph $G$ is a proper edge-coloring of $G$ with colors $1,\ldots,t$ such that all colors are used and for each $v\in V(G)$, the set of colors of the edges incident to $v$ is an interval of integers. A graph $G$ is \emph{interval colorable} if it has an interval $t$-coloring for some positive integer $t$. The set of all interval colorable graphs is denoted by $\mathfrak{N}$.

The notion of interval colorings was introduced by Asratian and Kamalian \cite{AsrKam} (available in English as \cite{AsrKamJCTB}) in 1987 and was motivated by the problem of finding compact school timetables, that is, timetables such that the lectures of each teacher and each class are scheduled at consecutive periods. This problem corresponds to the problem of finding an interval edge-coloring of a bipartite multigraph. 
In \cite{AsrKam,AsrKamJCTB}, Asratian and Kamalian noted that if $G\in \mathfrak{N}$, then $\chi^{\prime }\left(G\right)=\Delta(G)$. Moreover,
they also proved \cite{AsrKam,AsrKamJCTB} that if a triangle-free graph $G$ has an interval $t$-coloring, then $t\leq \left\vert V(G)\right\vert -1$. In \cite{Kampreprint}, Kamalian investigated interval colorings of complete bipartite graphs and trees. In particular, he proved that the complete bipartite graph
$K_{m,n}$ has an interval $t$-coloring if and only if
$m+n-\gcd(m,n)\leq t\leq m+n-1$, where $\gcd(m,n)$ is the greatest
common divisor of $m$ and $n$. In \cite{Petrosyan,PetrosyanKhachatrianTananyan}, Petrosyan, Khachatrian and Tananyan investigated interval colorings of complete graphs and
$n$-dimensional cubes. In particular, they proved that the $n$-dimensional cube $Q_{n}$ has an interval $t$-coloring if and only if $n\leq t\leq \frac{n\left(n+1\right)}{2}$. Generally, it is an $NP$-complete problem to determine whether a bipartite graph has an interval coloring \cite{Seva}.
In fact, for every positive integer $\Delta\geq 11$, there exists a bipartite graph with maximum degree $\Delta$ that has no interval coloring \cite{PetrosHrant}. However, some classes of graphs have been proved to admit interval colorings; it is known, for example, that trees, regular and complete bipartite graphs \cite{AsrKam,Hansen,Kampreprint}, subcubic graphs with $\chi^{\prime }\left(G\right)=\Delta(G)$ \cite{ArmenCarlJohanPetros}, doubly convex bipartite graphs \cite{AsrDenHag,KamDiss}, grids \cite{GiaroKubale1}, outerplanar bipartite graphs \cite{GiaroKubale2}, $(2,b)$-biregular graphs \cite{Hansen,HansonLotenToft,KamMir} and $(3,6)$-biregular graphs \cite{CarlJToft} have interval colorings, where an \emph{$(a,b)$-biregular} graph is a bipartite graph where the vertices in one part all have degree $a$ and the vertices in the other part all have degree $b$.

Interval edge-colorings of Cartesian products of graphs were first studied by Giaro and Kubale \cite{GiaroKubale1}. In \cite{GiaroKubale1}, Giaro and Kubale showed that all grids, bipartite cylinders and tori have interval colorings. In 2004, Giaro and Kubale \cite{GiaroKubale2} proved that if $G,H\in \mathfrak{N}$, then $G\square H\in \mathfrak{N}$. Moreover, $w(G\square H)\leq w(G) + w(H)$ and $W(G\square H)\geq W(G) + W(H)$. In 2011, Petrosyan \cite{Petrosyanproducts} investigated interval colorings of various products of graphs. In particular, he proved that the torus $C_{n_{1}}\square C_{n_{2}}\in \mathfrak{N}$ if and only if $n_{1}n_{2}$ is even, and the Hamming graph $K_{n_{1}}\square K_{n_{2}}\square\cdots \square K_{n_{k}}\in \mathfrak{N}$ if and only if $n_{1}n_{2}\cdots n_{k}$ is even.
In 2013, Petrosyan, Khachatrian and Tananyan \cite{PetrosyanKhachatrianTananyan} proved that if $G$ is an $r$-regular graph and $G\in \mathfrak{N}$, then $W(G\square P_{m})\geq W(G)+W(P_{m})+(m-1)r$ and $W(G\square C_{2n})\geq W(G)+W(C_{2n})+nr$ ($r\geq 2$). Moreover, they also showed that if $G\square H$ is planar and both factors have at least $3$ vertices, then $G\square H\in\mathfrak{N}$ and $w(G\square H)\leq 6$. Interval edge-colorings of some other products of graphs were considered in \cite{Kubale,CSITproducts,TepanyanPetrosyan}. 

In this paper we continue our study of interval edge-colorings of Cartesian products of graphs. We first consider interval colorability of Cartesian products of graphs depending on the interval colorability of the factors. We also improve the known lower bound on $W(G\square H)$ for interval colorable graphs $G$ and $H$, by showing that if $G,H\in \mathfrak{N}$ and $H$ is an $r$-regular graph, then $W(G\square H)\geq W(G)+W(H)+r$. Then we introduce the notion of separable interval colorings of graphs and use it for further improvements of the previous bounds. Next, we derive a new upper bound on $W(G)$ for interval colorable connected graphs with additional distance conditions. Based on these bounds, we improve known lower and upper bounds on $W(C_{2n_{1}}\square C_{2n_{2}}\square\cdots \square C_{2n_{k}})$ for $k$-dimensional tori $C_{2n_{1}}\square C_{2n_{2}}\square\cdots \square C_{2n_{k}}$ and on $W(K_{2n_{1}}\square K_{2n_{2}}\square\cdots \square K_{2n_{k}})$ for Hamming graphs $K_{2n_{1}}\square K_{2n_{2}}\square\cdots \square K_{2n_{k}}$, and these new bounds coincide with each other for hypercubes. Finally, we give several results on interval colorings of Fibonacci cubes $\Gamma_{n}$.  

\section{Notation, definitions and auxiliary results}

If $G$ is a connected graph, the distance between two vertices $u$ and $v$ in $G$, we denote by $d_{G}(u,v)$ (or $d(u,v)$), the eccentricity of a vertex $v$ in $G$ by $\epsilon(v)$ and the diameter of $G$ by $\mathrm{diam}(G)$. For a vertex $v\in V(G)$, we denote by $N_{i}(v)$, the set of vertices of $G$ at distance $i$ from $v$. The interval $I(u,v)$ between two vertices $u$ and $v$ of a connected graph $G$ is the set of vertices on shortest paths between $u$ and $v$. Clearly, $I(u,v)$ contains $u$ and $v$. We use the standard notations $P_{n}$, $C_{n}$, $K_{n}$ and $Q_{n}$ for the path, cycle, complete graph on
$n$ vertices and the hypercube of dimension $n$, respectively. A partial edge-coloring of a graph $G$ is a coloring of some of the edges of
$G$ such that no two adjacent edges receive the same color. If
$\alpha$ is a partial edge-coloring of $G$ and $v\in V(G)$, then
$S\left(v,\alpha\right)$ denotes the set of colors appearing on
colored edges incident to $v$. If $\alpha $ 
is a partial edge-coloring of a graph $G$ and $v\in V(G)$, then the
smallest and largest colors of $S\left(v,\alpha \right)$
are denoted by $\underline S\left(v,\alpha \right)$ and $\overline S\left(v,\alpha \right)$, respectively. Clearly, if $\alpha$ is a proper
edge-coloring of a graph $G$, then $\vert S(v,\alpha)\vert =d_G(v)$
for every $v\in V(G)$.

For two positive integers $a$ and $b$ with $a\leq b$, the set
$\left\{a,a+1,\ldots ,b\right\}$ is denoted by $\left[a,b\right]$.

Let $G$ and $H$ be graphs. The Cartesian product $G\square H$ is
defined as follows:
\begin{center}
$V(G\square H)=V(G)\times V(H)$,
\end{center}
\begin{center}
$E(G\square H)=\{(u_{1},v_{1})(u_{2},v_{2})\colon\,
(u_{1}=u_{2}~and~v_{1}v_{2}\in E(H))~or~(v_{1}=v_{2}~and~ u_{1}u_{2}\in
E(G))\}$.
\end{center}

Clearly, if $G$ and $H$ are connected graphs, then $G\square H$ is
connected, too. Moreover, $\Delta(G\square H)=\Delta(G)+\Delta(H)$
and $\mathrm{diam}(G\square H)=\mathrm{diam}(G)+\mathrm{diam}(H)$.

The $k$-dimensional torus $T(n_{1},\ldots,n_{k})$ ($n_{i}\in
\mathbb{N},n_{i}\geq 3$) is the Cartesian product of cycles $C_{n_{1}}\square
C_{n_{2}}\square\cdots\square C_{n_{k}}$. The Hamming graph $H(n_{1},\ldots,n_{k})$ ($n_{i}\in \mathbb{N}$) is the Cartesian product of complete graphs $K_{n_{1}}\square K_{n_{2}}\square\cdots\square K_{n_{k}}$. The graph $H_{n}^{k}$ is the Cartesian
product of the complete graph $K_{n}$ by itself $k$ times. Clearly, the graph $H_{2}^{n}$ is isomorphic to the $n$-dimensional cube $Q_{n}$. We use \cite{HammackImrichKlavzar} for concepts and notation on the Cartesian products of graphs not defined here.\\

We also need the following results.

\begin{lemma} \cite{KamDiss}
\label{ourlemma} If $\alpha$ is an edge-coloring of a connected
graph $G$ with colors $1,\ldots,t$ such that the edges incident to
each vertex $v\in V(G)$ are colored by distinct and consecutive
colors, and $\min_{e\in E(G)}\{\alpha(e)\}=1$, $\max_{e\in
E(G)}\{\alpha(e)\}=t$, then $\alpha$ is an interval $t$-coloring of
$G$.
\end{lemma}

\begin{theorem} \cite{CarlJohanHrantPetros}
\label{Eulerian} If $G$ is an Eulerian graph and $|E(G)|$ is odd, then $G\notin \mathfrak{N}$.
\end{theorem}

\begin{theorem} \cite{PetrosyanKhachatrianTananyan}
\label{2-Tori} 
For any integers $m,n \geq 2$, we have $W(T(2m,2n))\geq \max\{3m+n+2,3n+m+2\}$, and for any $m\geq 2$, $n\in\mathbb{N}$, we have 
\begin{center}
$W\left(T(2m,2n+1)\right)\geq \left\{
\begin{tabular}{ll}
$2m+2n+2$, & if $m$ is odd,\\
$2m+2n+3$, & if $m$ is even.\\
\end{tabular}%
\right.$
\end{center}
\end{theorem}

\begin{theorem} \cite{HrantPetros}
\label{Completegraph} If $n=\Pi_{i=1}^{\infty}p_{i}^{\alpha_{i}}$, where $p_{i}$ is the $i$th prime number and $\alpha_{i}\in \mathbb{Z}_{\geq 0}$, then
$$W(K_{2n}) \geq 4n - 3 - \alpha_1 - 2\alpha_2 - 3\alpha_3 - 4\alpha_4 - 4\alpha_5 - \frac{1}{2}\sum\limits_{i=6}^{\infty}{\alpha_i(p_i+1)}.$$
\end{theorem}
\bigskip	

\section{Interval colorability of Cartesian products of graphs}

Interval colorability of Cartesian products of graphs was first studied by Giaro and Kubale \cite{GiaroKubale1} in 1997, where the authors proved that if $G\in \mathfrak{N}$, then the Cartesian products $G\square P_{n}$ and $G\square C_{2n}$ ($n\geq 2$) are interval colorable. As a corollary they obtained that all $k$-dimensional grids $P_{n_{1}}\square P_{n_{2}}\square \cdots \square P_{n_{k}}$, bipartite cylinders $P_{n_{1}}\square C_{n_{2}}$ and bipartite tori $C_{n_{1}}\square C_{n_{2}}$ have interval colorings. In 2004, Giaro and Kubale \cite{GiaroKubale2} also proved the following result.

\begin{theorem}
\label{GiaroKubale}
If $G,H\in \mathfrak{N}$, then $G\square H\in \mathfrak{N}$.  
\end{theorem}

Later, Petrosyan \cite{Petrosyanproducts} observed that interval colorability of one of the factors of Cartesian products of regular graphs is sufficient for interval colorability of the Cartesian products of regular graphs. More precisely, the following result holds.

\begin{theorem}
\label{Petrosyan}
If $G$ and $H$ are two regular graphs for which at least one of
the following conditions holds:
\begin{description}
    \item[a)] $G$ and $H$ contain a perfect matching,
    \item[b)] $G\in \mathfrak{N}$,
    \item[c)] $H\in \mathfrak{N}$, 
\end{description}  
then $G\square H\in \mathfrak{N}$ and $w(G\square H)=\Delta (G\square H)$.
\end{theorem}

We now show that for non-regular graphs interval colorability of one of the factors of Cartesian products of non-regular graphs is not sufficient for interval colorability of the Cartesian products of these graphs.

\begin{proposition}
\label{interval non-colorable}
If $G$ is an Eulerian graph with an odd number of vertices and an even number of edges, and $H$ is an Eulerian graph with an odd number of edges, then $G\square H\notin \mathfrak{N}$.
\end{proposition}
\begin{proof}
Clearly, if $G$ and $H$ are Eulerian, then $G\square H$ is Eulerian too.
Since $G$ has an odd number of vertices, even number of edges and $H$ has an odd number of edges, we have
$$|E(G\square H)|=|V(G)||E(H)|+|V(H)||E(G)|\text{~is~odd}.$$

Thus, by Theorem \ref{Eulerian}, $G\square H\notin \mathfrak{N}$.
\end{proof}

If we take two copies of an arbitrary Eulerian graph and gluing any vertex of the first copy of the graph with any vertex of the second copy of the graph, then the resulting graph can be considered as the graph $G$ of Proposition \ref{interval non-colorable}. In particular, an example of such a graph is a butterfly graph (the graph obtained by gluing two copies of $K_{3}$ at a vertex), which is interval colorable. As a graph $H$ of Proposition \ref{interval non-colorable} one can take, for example, the triangle $K_{3}$, which is not interval colorable. So, the Cartesian product of the butterfly graph $G$ and the triangle $H$, according to Proposition \ref{interval non-colorable}, is not interval colorable. (Tab. 1)

The Cartesian product of two interval non-colorable graphs can be either interval colorable or interval non-colorable. In particular, the Petersen graph $P_{10}$ with $\chi'(P_{10})=4$ does not satisfy the necessary condition of the interval colorability, which states that if $G\in \mathfrak{N}$, then $\chi^{\prime }\left(G\right)=\Delta(G)$, therefore it is not interval colorable, but the Cartesian product of two such graphs $P_{10}\square P_{10}$ is interval colorable, according to the first point of Theorem \ref{Petrosyan}. On the other hand, if we consider the Cartesian product of two odd cycles $C_{2m+1}\square C_{2n+1}$, then, by Theorem \ref{Eulerian}, $C_{2m+1}\notin \mathfrak{N}$ and $C_{2n+1}\notin \mathfrak{N}$, but in this case $C_{2m+1}\square C_{2n+1}$ is not interval colorable, since $\chi'(C_{2m+1}\square C_{2n+1})=\Delta(C_{2m+1}\square C_{2n+1})+1=5$.

\begin{table}[h]
    \centering
    \begin{tabular}{c|c|c}
         & $G \square H \in \mathfrak{N}$ & $G \square H \notin \mathfrak{N}$ \\
         \hline
        $G \in \mathfrak{N}$, $H \in \mathfrak{N}$ & $G=H=K_2$ & impossible \\
        $G \in \mathfrak{N}$, $H \notin \mathfrak{N}$ & $G=K_2$, $H=K_3$ & $G=$\begin{tikzpicture}[baseline=-0.5ex]
        \draw[fill=black] (0:0) circle(1pt);
        \draw[fill=black] (30:0.5cm) circle(1pt);
        \draw[fill=black] (-30:0.5cm) circle(1pt);
        \draw[fill=black] (150:0.5cm) circle(1pt);
        \draw[fill=black] (-150:0.5cm) circle(1pt);
        \draw (30:0.5cm) -- (0:0) -- (150:0.5cm) -- (-150:0.5cm) -- (0:0) -- (-30:0.5cm) -- cycle;
        \end{tikzpicture}, $H=K_3$ \\ 
        $G \notin \mathfrak{N}$, $H \notin \mathfrak{N}$ & $G=H=P_{10}$ & $G=H=C_3$ \\ 
        
    \end{tabular}
    \caption{Interval colorability of Cartesian products of graphs depending on the factors of the product}
    \label{table-Cartesian}
\end{table}

\section{Cartesian products of interval colorable graphs and separable interval colorings of graphs}

As we mentioned before, Giaro and Kubale \cite{GiaroKubale2} proved that if $G,H\in \mathfrak{N}$, then $G\square H\in \mathfrak{N}$. From the proof of the result, it also follows that $w(G\square H)\leq w(G)+w(H)$ and $W(G\square H)\geq W(G)+W(H)$. Here we show that the lower bound on $W(G\square H)$ can be improved when $H$ is a regular graph.

\begin{theorem}
\label{th:general}
If $G,H\in \mathfrak{N}$ and $H$ is an $r$-regular graph, then 
$$W(G\square H)\geq W(G)+W(H)+r.$$
\end{theorem}
\begin{proof}
Let $V(G)=\{u_{1},u_{2},\ldots,u_{m}\}$ and $V(H)=\{v_{1},v_{2},\ldots,v_{n}\}$. Also, let $\alpha$ be an interval $W(G)$-coloring of $G$ and $\beta$ be an interval $W(H)$-coloring of $H$, respectively. Without loss of generality we may assume that $\underline S\left(u_{1},\alpha \right)=1, \underline S\left(v_{1},\beta \right)=1$ and 
$\overline S\left(u_{m},\alpha \right)=W(G), \overline S\left(v_{n},\beta \right)=W(H)$. For the proof, we need to construct an interval $(W(G)+W(H)+r)$-coloring of $G\square H$. 

Let us define an edge-coloring $\gamma$ of $G\square H$ as follows:
    
\begin{description}
    \item[(1)] for each $i$ ($1 \leq i \leq m-1$) and $xy\in E(H)$, let 
$$\gamma\left((u_{i},x)(u_{i},y)\right)=\beta(xy)+\underline S\left(u_{i},\alpha \right)-1;$$
    \item[(2)] for each $j$ ($1 \leq j \leq n$) and $xy\in E(G)$, let 
$$\gamma\left((x,v_{j})(y,v_{j})\right)=\alpha(xy)+\overline S\left(v_{j},\beta\right);$$   
    \item[(3)] for each $xy\in E(H)$, let 
$$\gamma\left((u_{m},x)(u_{m},y)\right)=\beta(xy)+\overline S\left(u_{m},\alpha\right)+r.$$   
\end{description}   

Let us show that $\gamma$ is an interval $(W(G)+W(H)+r)$-coloring of the graph $G\square H$.

By the definition of $\gamma$ and taking into account that $\overline S\left(v_{j},\beta \right)-\underline S\left(v_{j},\beta \right)=r-1$ for $1\leq j\leq n$, we have

\begin{itemize}

\item for $1\leq i\leq m-1$ and $1\leq j\leq n$,

\begin{eqnarray*}
S\left((u_{i},v_{j}),\gamma\right) &=&\left[\underline S\left(v_{j},\beta \right)+\underline S\left(u_{i},\alpha\right)-1,\overline S\left(v_{j},\beta \right)+\underline S\left(u_{i},\alpha\right)-1\right]\cup\\ 
&\cup &\left[\underline S\left(u_{i},\alpha\right)+\overline S\left(v_{j},\beta\right),\overline S\left(u_{i},\alpha\right)+\overline S\left(v_{j},\beta\right)\right]=\\
&=& \left[\underline S\left(v_{j},\beta \right)+\underline S\left(u_{i},\alpha\right)-1,\overline S\left(u_{i},\alpha\right)+\overline S\left(v_{j},\beta\right)\right],
\end{eqnarray*}

\item for $1\leq j\leq n$,

\begin{eqnarray*}
S\left((u_{m},v_{j}),\gamma\right) &=&\left[\underline S\left(u_{m},\alpha \right)+\overline S\left(v_{j},\beta\right),\overline S\left(u_{m},\alpha \right)+\overline S\left(v_{j},\beta\right)\right]\cup\\ 
&\cup &\left[\underline S\left(v_{j},\beta\right)+\overline S\left(u_{m},\alpha\right)+r,\overline S\left(v_{j},\beta\right)+\overline S\left(u_{m},\alpha\right)+r\right]=\\
&=& \left[\underline S\left(u_{m},\alpha\right)+\overline S\left(v_{j},\beta\right),\overline S\left(v_{j},\beta\right)+\overline S\left(u_{m},\alpha\right)+r\right].
\end{eqnarray*}

\end{itemize}

This shows that $\gamma$ is a proper edge-coloring of $G\square H$ such that for each $(u,v)\in V(G\square H)$, $S((u,v),\gamma)$ is an interval of integers. Since $\underline S\left((u_{1},v_{1}),\gamma\right)=\underline S\left(v_{1},\beta \right)+\underline S\left(u_{1},\alpha \right)-1=1$ and 
$\overline S\left((u_{m},v_{n}),\gamma\right)=\overline S\left(v_{n},\beta \right)+\overline S\left(u_{m},\alpha \right)+r=W(G)+W(H)+r$, by Lemma \ref{ourlemma}, $\gamma$ is an interval $(W(G)+W(H)+r)$-coloring of $G\square H$. 
\end{proof}

\begin{corollary}\label{twographs}
If $G$ is an $r$-regular graph, $H$ is an $r'$-regular graph and $G,H\in \mathfrak{N}$, then 
$$W(G\square H)\geq W(G)+W(H)+\max\{r,r'\}.$$
\end{corollary}

Since the Cartesian product of two regular graphs is also regular, iteratively applying Corollary \ref{twographs}, we obtain the following result.

\begin{corollary}\label{manygraph}
If $G_{i}$ is an $r_{i}$-regular graph ($1\leq i\leq k$, $r_{1}\geq r_{2}\geq \cdots \geq r_{k}$) and $G_{1},G_{2},\ldots, G_{k}\in \mathfrak{N}$, then 
$$W(G_{1}\square G_{2}\square \cdots \square G_{k})\geq \sum_{i=1}^{k}W(G_{i})+\sum_{i=1}^{k-1}\sum_{j=1}^{i}r_{j}.$$
\end{corollary}

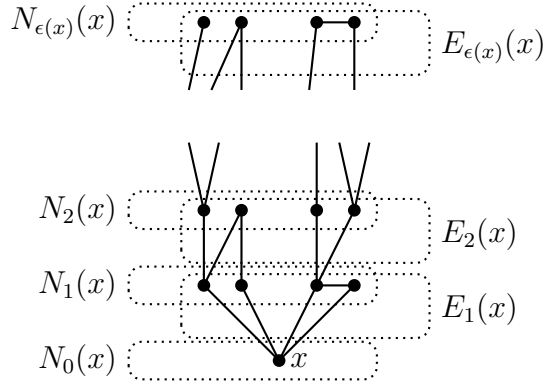
\begin{figure}[t]
\centering
\begin{tikzpicture}[style=thick]
    \coordinate (V0) at (2cm,1cm);
    \coordinate (V11) at (1cm,2cm);
    \coordinate (V12) at (1.5cm,2cm);
    \coordinate (V13) at (2.5cm,2cm);
    \coordinate (V14) at (3cm,2cm);
    \coordinate (V21) at (1cm,3cm);
    \coordinate (V22) at (1.5cm,3cm);
    \coordinate (V23) at (2.5cm,3cm);
    \coordinate (V24) at (3cm,3cm);
    \coordinate (V31) at (0.8cm,3.9cm);
    \coordinate (V32) at (1.2cm,3.9cm);
    \coordinate (V33) at (2.5cm,3.9cm);
    \coordinate (V34) at (2.8cm,3.9cm);
    \coordinate (V35) at (3.2cm,3.9cm);
    \coordinate (V41) at (0.8cm,4.6cm);
    \coordinate (V42) at (1.1cm,4.6cm);
    \coordinate (V43) at (1.5cm,4.6cm);
    \coordinate (V44) at (2.4cm,4.6cm);
    \coordinate (V45) at (3cm,4.6cm);
    \coordinate (Vn1) at (1cm,5.5cm);
    \coordinate (Vn2) at (1.5cm,5.5cm);
    \coordinate (Vn3) at (2.5cm,5.5cm);
    \coordinate (Vn4) at (3cm,5.5cm);
    
    \draw (V0) -- (V11);
    \draw (V0) -- (V12);
    \draw (V0) -- (V13);
    \draw (V0) -- (V14);
    \draw (V13) -- (V14);
    
    \draw (V11) -- (V21);
    \draw (V11) -- (V22);
    \draw (V12) -- (V22);
    \draw (V13) -- (V23);
    \draw (V13) -- (V24);
    
    \draw (V21) -- (V31);
    \draw (V21) -- (V32);
    \draw (V23) -- (V33);
    \draw (V24) -- (V34);
    \draw (V24) -- (V35);
    
    \draw (V41) -- (Vn1);
    \draw (V42) -- (Vn2);
    \draw (V43) -- (Vn2);
    \draw (V44) -- (Vn3);
    \draw (V45) -- (Vn4);
    \draw (Vn3) -- (Vn4);
    
    \draw[dotted,rounded corners=5pt]
  (0,1) node[left]{$N_0(x)$} ++(0,-0.25) rectangle ++(3.3,0.5) ;
  
  \draw[dotted,rounded corners=5pt]
  (0,2) node[left]{$N_1(x)$} ++(0,-0.25) rectangle ++(3.3,0.5) ;
  \draw[dotted,rounded corners=5pt]
  (0.7,1.3) rectangle ++(3.3,0.85) ++(0,-0.47) node[right]{$E_1(x)$};
  
  \draw[dotted,rounded corners=5pt]
  (0,3) node[left]{$N_2(x)$} ++(0,-0.25) rectangle ++(3.3,0.5) ;
  \draw[dotted,rounded corners=5pt]
  (0.7,2.3) rectangle ++(3.3,0.85) ++(0,-0.47) node[right]{$E_2(x)$};
  
  \draw[dotted,rounded corners=5pt]
  (0.7,4.8) rectangle ++(3.3,0.85) ++(0,-0.47) node[right]  {$E_{\epsilon(x)}(x)$} ;
  \draw[dotted,rounded corners=5pt]
  (0,5.5) node[left]{$N_{\epsilon(x)}(x)$} ++(0,-0.25) rectangle ++(3.3,0.5) ;
    
    \draw[fill=black] (V0) circle (2pt) node[right]{$x$};
    \draw[fill=black] (V11) circle (2pt);
    \draw[fill=black] (V12) circle (2pt);
    \draw[fill=black] (V13) circle (2pt);
    \draw[fill=black] (V14) circle (2pt);
    \draw[fill=black] (V21) circle (2pt);
    \draw[fill=black] (V22) circle (2pt);
    \draw[fill=black] (V23) circle (2pt);
    \draw[fill=black] (V24) circle (2pt);
    \draw[fill=black] (Vn1) circle (2pt);
    \draw[fill=black] (Vn2) circle (2pt);
    \draw[fill=black] (Vn3) circle (2pt);
    \draw[fill=black] (Vn4) circle (2pt);
\end{tikzpicture}
\caption{Level representation of a connected graph $G$ with respect to $x\in V(G)$.}
\label{graphLevels}
\end{figure}

Let us now consider a connected graph $G$. Choose a vertex $x\in V(G)$ and consider the level representation of $G$ with respect to $x$: 

$$V(G)=\bigcup_{i=0}^{\epsilon(x)}N_{i}(x),\text{~where}$$
$$N_{i}(x)=\{v\in V(G): d(v,x)=i\}~(0\leq i\leq \epsilon(x));$$
$$E(G)=\bigcup_{i=1}^{\epsilon(x)}E_{i}(G),\text{~where}$$
$$E_{i}(G)=\{uv\in E(G): u\in N_{i-1}(x)\cup N_{i}(x),v\in N_{i}(x)\}~(1\leq i\leq \epsilon(x)).\text{~(See Fig. 1)}$$

If $\alpha$ is a proper edge-coloring of a graph $G$ with the level representation with respect to $x$ defined before, then for each $v\in N_{i}(x)$ ($1\leq i\leq \epsilon(x)$), define the sets $S_{x}^{-}(v,\alpha)$ and $S_{x}^{+}(v,\alpha)$ as follows:
$$S(v,\alpha)=S_{x}^{-}(v,\alpha)\cup S_{x}^{+}(v,\alpha), \text{~where}$$
$$S_{x}^{-}(v,\alpha)=\{\alpha(uv):uv\in E(G), u\in N_{i-1}(x)\cup N_{i}(x)\},$$
$$S_{x}^{+}(v,\alpha)=\{\alpha(uv):uv\in E(G), u\in N_{i+1}(x)\}.$$

\begin{figure}[b!]
\centering
\begin{tikzpicture}[style=thick]
    \node at (-0.5cm, 1cm) {$N_{i-1}(x)$};
    \node at (-0.5cm, 2.5cm) {$N_i(x)$};
    \node at (-0.5cm, 4cm) {$N_{i+1}(x)$};
    
    \coordinate (V111) at (1cm,1cm);
    \coordinate (V112) at (1.5cm,1cm);
    \coordinate (V113) at (2cm,1cm);
    \coordinate (V121) at (1cm,2.5cm);
    \coordinate (V122) at (1.5cm,2.5cm);
    \coordinate (V123) at (2cm,2.5cm);
    \coordinate (V131) at (1cm,4cm);
    \coordinate (V132) at (1.5cm,4cm);
    \coordinate (V133) at (2cm,4cm);
    
    \draw (V111) -- (V122);
    \draw (V112) -- (V122);
    \draw (V113) -- (V122);
    \draw (V121) -- (V122);
    \draw (V123) -- (V122);
    \draw (V131) -- (V122);
    \draw (V132) -- (V122);
    \draw (V133) -- (V122);
    
    \draw[fill=black] (V111) circle (2pt);
    \draw[fill=black] (V112) circle (2pt);
    \draw[fill=black] (V113) circle (2pt);
    \draw[fill=black] (V121) circle (2pt);
    \draw[fill=black] (V122) circle (2pt) node[right] at (1.55cm,2.7cm) {$v$};
    \draw[fill=black] (V123) circle (2pt);
    \draw[fill=black] (V131) circle (2pt);
    \draw[fill=black] (V132) circle (2pt);
    \draw[fill=black] (V133) circle (2pt);
    
    \node at (1.5cm, 0.3cm) {$S(v,\alpha)$};
    
    \coordinate (V211) at (4cm,1cm);
    \coordinate (V212) at (4.5cm,1cm);
    \coordinate (V213) at (5cm,1cm);
    \coordinate (V221) at (4cm,2.5cm);
    \coordinate (V222) at (4.5cm,2.5cm);
    \coordinate (V223) at (5cm,2.5cm);
    \coordinate (V231) at (4cm,4cm);
    \coordinate (V232) at (4.5cm,4cm);
    \coordinate (V233) at (5cm,4cm);
    
    \draw (V211) -- (V222);
    \draw (V212) -- (V222);
    \draw (V213) -- (V222);
    \draw (V221) -- (V222);
    \draw (V223) -- (V222);
    \draw[dotted] (V231) -- (V222);
    \draw[dotted] (V232) -- (V222);
    \draw[dotted] (V233) -- (V222);
    
    \draw[fill=black] (V211) circle (2pt);
    \draw[fill=black] (V212) circle (2pt);
    \draw[fill=black] (V213) circle (2pt);
    \draw[fill=black] (V221) circle (2pt);
    \draw[fill=black] (V222) circle (2pt) node[right] at (4.55cm,2.7cm) {$v$};
    \draw[fill=black] (V223) circle (2pt);
    \draw[fill=black] (V231) circle (2pt);
    \draw[fill=black] (V232) circle (2pt);
    \draw[fill=black] (V233) circle (2pt);
    
    \node at (4.5cm, 0.3cm) {$S_x^{-}(v,\alpha)$};
    
    \coordinate (V311) at (7cm,1cm);
    \coordinate (V312) at (7.5cm,1cm);
    \coordinate (V313) at (8cm,1cm);
    \coordinate (V321) at (7cm,2.5cm);
    \coordinate (V322) at (7.5cm,2.5cm);
    \coordinate (V323) at (8cm,2.5cm);
    \coordinate (V331) at (7cm,4cm);
    \coordinate (V332) at (7.5cm,4cm);
    \coordinate (V333) at (8cm,4cm);
    
    \draw[dotted] (V311) -- (V322);
    \draw[dotted] (V312) -- (V322);
    \draw[dotted] (V313) -- (V322);
    \draw[dotted] (V321) -- (V322);
    \draw[dotted] (V323) -- (V322);
    \draw (V331) -- (V322);
    \draw (V332) -- (V322);
    \draw (V333) -- (V322);
    
    \draw[fill=black] (V311) circle (2pt);
    \draw[fill=black] (V312) circle (2pt);
    \draw[fill=black] (V313) circle (2pt);
    \draw[fill=black] (V321) circle (2pt);
    \draw[fill=black] (V322) circle (2pt) node[right] at (7.55cm,2.7cm) {$v$};
    \draw[fill=black] (V323) circle (2pt);
    \draw[fill=black] (V331) circle (2pt);
    \draw[fill=black] (V332) circle (2pt);
    \draw[fill=black] (V333) circle (2pt);
    
    \node at (7.5cm, 0.3cm) {$S_x^{+}(v,\alpha)$};
    
\end{tikzpicture}
\caption{Lower and upper spectrums of the vertex $v$}
\label{vertexSpectrums}
\end{figure}
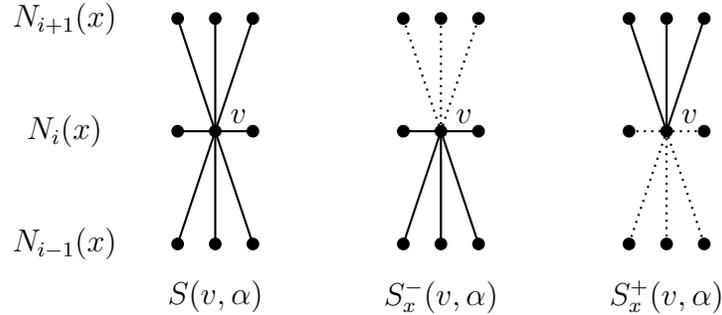

The sets $S_{x}^{-}(v,\alpha)$ and $S_{x}^{+}(v,\alpha)$
are called lower and upper spectrums of the vertex $v$, respectively (See Fig. 2). If $G$ is a connected graph and $\alpha$ its proper edge-coloring, then the vertex $v$ of $G$ is called a separable vertex with respect to $x$ if $\max S_{x}^{-}(v,\alpha)<\min S_{x}^{+}(v,\alpha)$. If $S_{x}^{-}(v,\alpha)=\emptyset$, then we set $\max S_{x}^{-}(v,\alpha)=-\infty$, and if $S_{x}^{+}(v,\alpha)=\emptyset$, then we set $\min S_{x}^{+}(v,\alpha)=+\infty$. Since $S_{x}^{-}(x,\alpha)=\emptyset$, the vertex $x$ is a separable vertex. An interval edge-coloring $\alpha$ of a connected graph $G$ is called a separable interval coloring with respect to $x\in V(G)$ if for each vertex $v\in V(G)$, $v$ is a separable vertex with respect to $x$.

\begin{theorem}
\label{th:separableintervalcoloring}
If a connected graph $G$ has a separable interval $t_{G}$-coloring $\alpha$ with respect to $x\in V(G)$ such that $1\in S(x,\alpha)$ and there exists a vertex $y\in N_{\epsilon(x)}(x)$ with $t_{G}\in S(y,\alpha)$, then for any $r$-regular graph $H$ with an interval $t_{H}$-coloring, we have 
$$W(G\square H)\geq t_{G}+t_{H}+\epsilon(x)r.$$
\end{theorem}
\begin{proof}
Let $\alpha$ be a separable interval $t_{G}$-coloring of $G$  with respect to $x\in V(G)$ and $\beta$ be an interval $t_{H}$-coloring of $H$, respectively. For the proof, we need to construct an interval $(t_{G}+t_{H}+\epsilon(x)r)$-coloring of $G\square H$. 

Let us define an edge-coloring $\gamma$ of $G\square H$ as follows:
    
\begin{description}
    \item[(1)] for any $vv'\in E(H)$, let 
$$\gamma\left((x,v)(x,v')\right)=\beta(vv');$$
    \item[(2)] for each $i$ ($0 \leq i \leq \epsilon(x)$) and $u\in N_{i}(x)$, $vv'\in E(H)$, let 
$$\gamma\left((u,v)(u,v')\right)=\beta(vv')+\max S_{x}^{-}(u,\alpha)+ir;$$   
    \item[(3)] for each $i$ ($1 \leq i \leq \epsilon(x)$) and $uu'\in E_{i}(x)$, $v\in V(H)$, let 
$$\gamma\left((u,v)(u',v)\right)=\alpha(uu')+\overline S\left(v,\beta\right)+ir-r.$$   
\end{description}   

Let us show that $\gamma$ is an interval $(t_{G}+t_{H}+\epsilon(x)r)$-coloring of the graph $G\square H$.

For each vertex $(u,v)\in V(G\square H)$, decompose the spectrum $S((u,v),\gamma)$ into three sets as follows:

$$S((u,v),\gamma)=S_{x}^{-}((u,v),\gamma)\cup S_{x}^{0}((u,v),\gamma)\cup S_{x}^{+}((u,v),\gamma), \text{~where}$$
$$S_{x}^{0}((u,v),\gamma)=\{\gamma((u,v)(u,v')): vv'\in E(H)\},$$
$$S_{x}^{+}((u,v),\gamma)=\{\gamma((u,v)(u',v)): uu'\in E_{i+1}(x)\},$$
$$S_{x}^{-}((u,v),\gamma)=\{\gamma((u,v)(u',v)): uu'\in E_{i}(x)\},$$
$$u\in N_{i}(x), i=0,1,\ldots,\epsilon(x)\text{~(See Fig. 3).}$$

We set $S_{x}^{-}((x,v),\gamma)=\emptyset$, and if $u\in N_{\epsilon(x)}(x)$, then we set $S_{x}^{+}((u,v),\gamma)=\emptyset$.

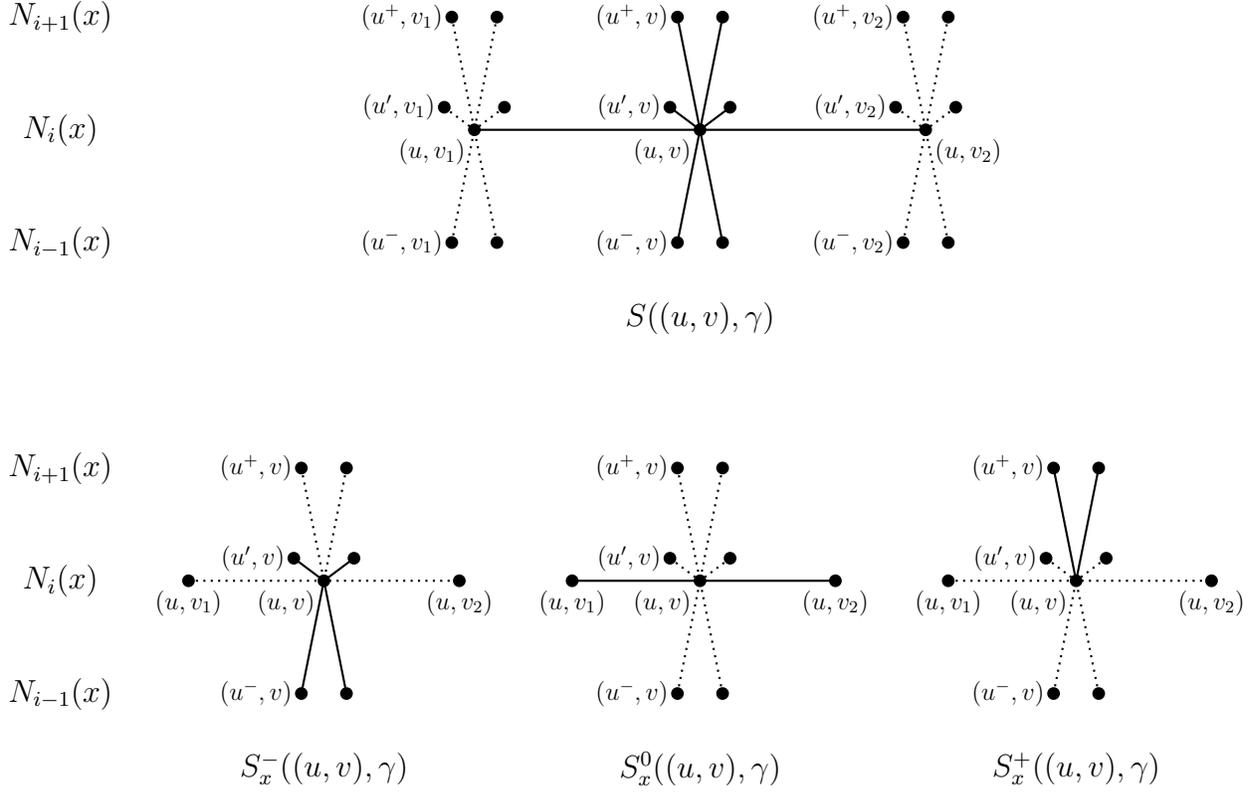
\begin{figure}[t!]
\centering
\begin{tikzpicture}[style=thick]
    \node at (-4cm, 1cm) {$N_{i-1}(x)$};
    \node at (-4cm, 2.5cm) {$N_i(x)$};
    \node at (-4cm, 4cm) {$N_{i+1}(x)$};
    \node at (-4cm, 7cm) {$N_{i-1}(x)$};
    \node at (-4cm, 8.5cm) {$N_i(x)$};
    \node at (-4cm, 10cm) {$N_{i+1}(x)$};
    
    \coordinate (Va111) at (1.2cm,7cm);
    \coordinate (Va113) at (1.8cm,7cm);
    \coordinate (Va121) at (1.1cm,8.8cm);
    \coordinate (Va122) at (1.5cm,8.5cm);
    \coordinate (Va123) at (1.9cm,8.8cm);
    \coordinate (Va131) at (1.2cm,10cm);
    \coordinate (Va133) at (1.8cm,10cm);
    
    \draw[dotted] (Va111) -- (Va122);
    \draw[dotted] (Va113) -- (Va122);
    \draw[dotted] (Va121) -- (Va122);
    \draw[dotted] (Va123) -- (Va122);
    \draw[dotted] (Va131) -- (Va122);
    \draw[dotted] (Va133) -- (Va122);
    
    \draw[fill=black] (Va111) circle (2pt) node[left,scale=0.8] {$(u^-,v_1)$};
    \draw[fill=black] (Va113) circle (2pt);
    \draw[fill=black] (Va121) circle (2pt) node[left,scale=0.8] {$(u',v_1)$};
    \draw[fill=black] (Va122) circle (2pt) node[below left,scale=0.8] {$(u,v_1)$};
    \draw[fill=black] (Va123) circle (2pt);
    \draw[fill=black] (Va131) circle (2pt) node[left,scale=0.8] {$(u^+,v_1)$};
    \draw[fill=black] (Va133) circle (2pt);
    
    \coordinate (Va211) at (4.2cm,7cm);
    \coordinate (Va213) at (4.8cm,7cm);
    \coordinate (Va221) at (4.1cm,8.8cm);
    \coordinate (Va222) at (4.5cm,8.5cm);
    \coordinate (Va223) at (4.9cm,8.8cm);
    \coordinate (Va231) at (4.2cm,10cm);
    \coordinate (Va233) at (4.8cm,10cm);
    
    \draw (Va211) -- (Va222);
    \draw (Va213) -- (Va222);
    \draw (Va221) -- (Va222);
    \draw (Va223) -- (Va222);
    \draw (Va231) -- (Va222);
    \draw (Va233) -- (Va222);
    
    \draw[fill=black] (Va211) circle (2pt) node[left,scale=0.8] {$(u^-,v)$};
    \draw[fill=black] (Va213) circle (2pt);
    \draw[fill=black] (Va221) circle (2pt) node[left,scale=0.8] {$(u',v)$};
    \draw[fill=black] (Va222) circle (2pt) node[below left,scale=0.8] {$(u,v)$};
    \draw[fill=black] (Va223) circle (2pt);
    \draw[fill=black] (Va231) circle (2pt) node[left,scale=0.8] {$(u^+,v)$};
    \draw[fill=black] (Va233) circle (2pt);
    
    \coordinate (Va311) at (7.2cm,7cm);
    \coordinate (Va313) at (7.8cm,7cm);
    \coordinate (Va321) at (7.1cm,8.8cm);
    \coordinate (Va322) at (7.5cm,8.5cm);
    \coordinate (Va323) at (7.9cm,8.8cm);
    \coordinate (Va331) at (7.2cm,10cm);
    \coordinate (Va333) at (7.8cm,10cm);
    
    \draw[dotted] (Va311) -- (Va322);
    \draw[dotted] (Va313) -- (Va322);
    \draw[dotted] (Va321) -- (Va322);
    \draw[dotted] (Va323) -- (Va322);
    \draw[dotted] (Va331) -- (Va322);
    \draw[dotted] (Va333) -- (Va322);
    
    \draw[fill=black] (Va311) circle (2pt) node[left,scale=0.8] {$(u^-,v_2)$};
    \draw[fill=black] (Va313) circle (2pt);
    \draw[fill=black] (Va321) circle (2pt) node[left,scale=0.8] {$(u',v_2)$};
    \draw[fill=black] (Va322) circle (2pt) node[below right,scale=0.8] {$(u,v_2)$};
    \draw[fill=black] (Va323) circle (2pt);
    \draw[fill=black] (Va331) circle (2pt) node[left,scale=0.8] {$(u^+,v_2)$};
    \draw[fill=black] (Va333) circle (2pt);
    
    \draw (Va122) -- (Va222) -- (Va322);
    
    \node at (4.5cm, 6cm) {$S((u,v),\gamma)$};
    
    \coordinate (Vc211) at (4.2cm,1cm);
    \coordinate (Vc213) at (4.8cm,1cm);
    \coordinate (Vc221) at (4.1cm,2.8cm);
    \coordinate (Vc222) at (4.5cm,2.5cm);
    \coordinate (Vc223) at (4.9cm,2.8cm);
    \coordinate (Vc231) at (4.2cm,4cm);
    \coordinate (Vc233) at (4.8cm,4cm);
    
    \draw[dotted] (Vc211) -- (Vc222);
    \draw[dotted] (Vc213) -- (Vc222);
    \draw[dotted] (Vc221) -- (Vc222);
    \draw[dotted] (Vc223) -- (Vc222);
    \draw[dotted] (Vc231) -- (Vc222);
    \draw[dotted] (Vc233) -- (Vc222);
    
    \draw[fill=black] (Vc211) circle (2pt) node[left,scale=0.8] {$(u^-,v)$};
    \draw[fill=black] (Vc213) circle (2pt);
    \draw[fill=black] (Vc221) circle (2pt) node[left,scale=0.8] {$(u',v)$};
    \draw[fill=black] (Vc222) circle (2pt) node[below left,scale=0.8] {$(u,v)$};
    \draw[fill=black] (Vc223) circle (2pt);
    \draw[fill=black] (Vc231) circle (2pt) node[left,scale=0.8] {$(u^+,v)$};
    \draw[fill=black] (Vc233) circle (2pt);
    
    \coordinate (Vc322) at (6.3cm,2.5cm);
    \coordinate (Vc122) at (2.8cm,2.5cm);
    \draw[fill=black] (Vc322) circle (2pt) node[below,scale=0.8] {$(u,v_2)$};
    \draw[fill=black] (Vc122) circle (2pt) node[below,scale=0.8] {$(u,v_1)$};
    \draw (Vc122) -- (Vc222) -- (Vc322);
    
    \node at (4.5cm, 0cm) {$S_x^0((u,v),\gamma)$};
    
    \coordinate (Vb211) at (-0.8cm,1cm);
    \coordinate (Vb213) at (-0.2cm,1cm);
    \coordinate (Vb221) at (-0.9cm,2.8cm);
    \coordinate (Vb222) at (-0.5cm,2.5cm);
    \coordinate (Vb223) at (-0.1cm,2.8cm);
    \coordinate (Vb231) at (-0.8cm,4cm);
    \coordinate (Vb233) at (-0.2cm,4cm);
    
    \draw (Vb211) -- (Vb222);
    \draw (Vb213) -- (Vb222);
    \draw (Vb221) -- (Vb222);
    \draw (Vb223) -- (Vb222);
    \draw[dotted] (Vb231) -- (Vb222);
    \draw[dotted] (Vb233) -- (Vb222);
    
    \draw[fill=black] (Vb211) circle (2pt) node[left,scale=0.8] {$(u^-,v)$};
    \draw[fill=black] (Vb213) circle (2pt);
    \draw[fill=black] (Vb221) circle (2pt) node[left,scale=0.8] {$(u',v)$};
    \draw[fill=black] (Vb222) circle (2pt) node[below left,scale=0.8] {$(u,v)$};
    \draw[fill=black] (Vb223) circle (2pt);
    \draw[fill=black] (Vb231) circle (2pt) node[left,scale=0.8] {$(u^+,v)$};
    \draw[fill=black] (Vb233) circle (2pt);
    
    \coordinate (Vb322) at (1.3cm,2.5cm);
    \coordinate (Vb122) at (-2.3cm,2.5cm);
    \draw[fill=black] (Vb322) circle (2pt) node[below,scale=0.8] {$(u,v_2)$};
    \draw[fill=black] (Vb122) circle (2pt) node[below,scale=0.8] {$(u,v_1)$};
    \draw[dotted] (Vb122) -- (Vb222) -- (Vb322);
    
    \node at (-0.5cm, 0cm) {$S_x^-((u,v),\gamma)$};
    
    \coordinate (Vd211) at (9.2cm,1cm);
    \coordinate (Vd213) at (9.8cm,1cm);
    \coordinate (Vd221) at (9.1cm,2.8cm);
    \coordinate (Vd222) at (9.5cm,2.5cm);
    \coordinate (Vd223) at (9.9cm,2.8cm);
    \coordinate (Vd231) at (9.2cm,4cm);
    \coordinate (Vd233) at (9.8cm,4cm);
    
    \draw[dotted] (Vd211) -- (Vd222);
    \draw[dotted] (Vd213) -- (Vd222);
    \draw[dotted] (Vd221) -- (Vd222);
    \draw[dotted] (Vd223) -- (Vd222);
    \draw (Vd231) -- (Vd222);
    \draw (Vd233) -- (Vd222);
    
    \draw[fill=black] (Vd211) circle (2pt) node[left,scale=0.8] {$(u^-,v)$};
    \draw[fill=black] (Vd213) circle (2pt);
    \draw[fill=black] (Vd221) circle (2pt) node[left,scale=0.8] {$(u',v)$};
    \draw[fill=black] (Vd222) circle (2pt) node[below left,scale=0.8] {$(u,v)$};
    \draw[fill=black] (Vd223) circle (2pt);
    \draw[fill=black] (Vd231) circle (2pt) node[left,scale=0.8] {$(u^+,v)$};
    \draw[fill=black] (Vd233) circle (2pt);
    
    \coordinate (Vd322) at (11.3cm,2.5cm);
    \coordinate (Vd122) at (7.8cm,2.5cm);
    \draw[fill=black] (Vd322) circle (2pt) node[below,scale=0.8] {$(u,v_2)$};
    \draw[fill=black] (Vd122) circle (2pt) node[below,scale=0.8] {$(u,v_1)$};
    \draw[dotted] (Vd122) -- (Vd222) -- (Vd322);
    
    \node at (9.5cm, 0cm) {$S_x^+((u,v),\gamma)$};
\end{tikzpicture}
\caption{Decomposition of the spectrum $S((u,v),\gamma)$ into three sets.}
\label{separable}
\end{figure}

By the definition of $\gamma$ and taking into account that $\overline S\left(v,\beta \right)-\underline S\left(v,\beta \right)=r-1$ for any $v\in V(H)$, we have

\begin{eqnarray*}
S_{x}^{0}((x,v),\gamma) &=& \{\beta(vv'): vv'\in E(H)\} =
\left[\underline S\left(v,\beta \right),\overline S\left(v,\beta \right)\right],\\ 
S_{x}^{0}((u,v),\gamma) &=& \{\beta(vv')+\max S_{x}^{-}(u,\alpha)+ir: vv'\in E(H)\} =\\
&=&\left[\underline S\left(v,\beta \right)+\max S_{x}^{-}(u,\alpha)+ir,\overline S\left(v,\beta \right)+\max S_{x}^{-}(u,\alpha)+ir\right].
\end{eqnarray*}

By the definition of $\gamma$ and taking into account that $\alpha$ is a separable interval coloring of $G$ with respect to $x$, we have

\begin{eqnarray*}
S_{x}^{+}((u,v),\gamma) &=& \{\alpha(uu')+\overline S\left(v,\beta \right)+(i+1)r-r: uu'\in E_{i+1}(x)\} =\\
&=& \left[\min S_{x}^{+}(u,\alpha)+\overline S\left(v,\beta \right)+ir,\max S_{x}^{+}(u,\alpha)+\overline S\left(v,\beta \right)+ir\right],\\ 
S_{x}^{-}((u,v),\gamma) &=& \{\alpha(uu')+\overline S\left(v,\beta \right)+ir-r: uu'\in E_{i}(x)\} =\\
&=&\left[\min S_{x}^{-}(u,\alpha)+\overline S\left(v,\beta \right)+ir-r,\max S_{x}^{-}(u,\alpha)+\overline S\left(v,\beta \right)+ir-r\right].
\end{eqnarray*}

Since $1\in S(x,\alpha)$, we have $\min S_{x}^{-}(x,\alpha)=1$ and for any $v\in V(H)$,

$$S((x,v),\gamma)=\left[\underline S\left(v,\beta \right),\max S_{x}^{+}(x,\alpha)+\overline S\left(v,\beta \right)\right].$$

Since $\underline S\left(v,\beta \right)-(\overline S\left(v,\beta \right)-r)=1$ and $\min S_{x}^{+}(u,\alpha)-\max S_{x}^{-}(u,\alpha)=1$, we have that for any $u\in V(G)$ ($u\neq x$) and $v\in V(H)$,

$$S((u,v),\gamma)=\left[\min S_{x}^{-}(u,\alpha)+\overline S\left(v,\beta \right)+ir-r,\max S_{x}^{+}(u,\alpha)+\overline S\left(v,\beta \right)+ir\right].$$

This shows that $\gamma$ is a proper edge-coloring of $G\square H$ such that for each $(u,v)\in V(G\square H)$, $S((u,v),\gamma)$ is an interval of integers. Since $\beta$ is an interval $t_{H}$-coloring of $H$, there are $v',v''\in V(H)$ such that $1\in S(v',\beta)$ and $t_{H}\in S(v'',\beta)$. On the other hand, since $1\in S(x,\alpha)$ and there exists a vertex $y\in N_{\epsilon(x)}(x)$ with $t_{G}\in S(y,\alpha)$, we obtain

$$1\in S((x,v'),\gamma)\text{~and~}t_{G}+t_{H}+\epsilon(x)r\in S((y,v''),\gamma).$$

From this and by Lemma \ref{ourlemma}, $\gamma$ is an interval $(t_{G}+t_{H}+\epsilon(x)r)$-coloring of $G\square H$. 
\end{proof}

\begin{corollary}\label{manygraphs}
If $G$ is an even cycle, or simple path, or $n$-dimesional cube, or caterpillar tree, or complete bipartite graph, and $H$ is an interval colorable $r$-regular graph, then
$$W(G\square H)\geq W(G)+W(H)+\mathrm{diam}(G)r.$$
In particular, we have:
\begin{description}
    \item[(a)] $$W(C_{2n}\square H)\geq n(r+1)+W(H)+1;$$
    \item[(b)] $$W(P_{n}\square H)\geq (n-1)(r+1)+W(H);$$ 
    \item[(c)] $$W(Q_{n}\square H)\geq \frac{n(n+2r+1)}{2}+W(H);$$
    \item[(d)] $$W(T\square H)\geq |E(T)|+W(H)+\mathrm{diam}(T)r;$$
    \item[(e)] $$W(K_{m,n}\square H)\geq m+n+2r+W(H)-1.$$
\end{description} 
\end{corollary}
\begin{proof}
For the proof, we need to show that graphs from the mentioned classes are satisfied Theorem \ref{th:separableintervalcoloring} conditions.

\textbf{(a). $\mathbf{G=C_{2n}}$}

Let $V(C_{2n})=\left\{x,v_1,\ldots,v_{n-1},y,u_{n-1},\ldots,u_1\right\}$. It is easy to see that an interval $W(C_{2n})$-coloring of $C_{2n}$ can be defined as follows:

\bigskip
\begin{tabular}{lll}
$\alpha(xv_1)=1$ & $\alpha(xu_1)=2$ \\
$\alpha(v_iv_{i+1})=i+1$ &$\alpha(u_iu_{i+1})=i+2$ & where $i=1,\ldots,n-2$\\
$\alpha(v_{n-1}y)=n$ &$\alpha(u_{n-1}y)=n+1$\\
\end{tabular}
\bigskip

Let us now consider the level representation of $C_{2n}$ with respect to $x$. In this case lower and upper spectrums of the vertices have the following form:

\bigskip
\begin{tabular}{lll}
$S_x^-(v_i,\alpha)=\left\{i\right\}$ &$S_x^+(v_i,\alpha)=\left\{i+1\right\}$ & where  $i=1,\ldots,n-2$\\
$S_x^-(u_i,\alpha)=\left\{i+1\right\}$ &$S_x^+(u_i,\alpha)=\left\{i+2\right\}$ & where  $i=1,\ldots,n-2$\\
\end{tabular}
\bigskip

Hence,

\bigskip
\begin{tabular}{lll}
$i = \max{S_x^-(v_i,\alpha)} < \min{S_x^+(v_i,\alpha)} = i+1$ & where $i=1,\ldots,n-2$\\
$i+1 = \max{S_x^-(u_i,\alpha)} < \min{S_x^+(u_i,\alpha)} = i+2$ & where  $i=1,\ldots,n-2$\\
$S_x^+(y,\alpha)=\emptyset$
\end{tabular}
\bigskip

This implies that $\alpha$ is a separable interval coloring of $C_{2n}$ with respect to $x$. Moreover, $1\in S(x,\alpha)$, $W(C_{2n})=n+1 \in S(y,\alpha)$ and $d(x,y)=\epsilon(x)=\mathrm{diam}(C_{2n})=n$. So, by Theorem \ref{th:separableintervalcoloring}, we obtain the result.

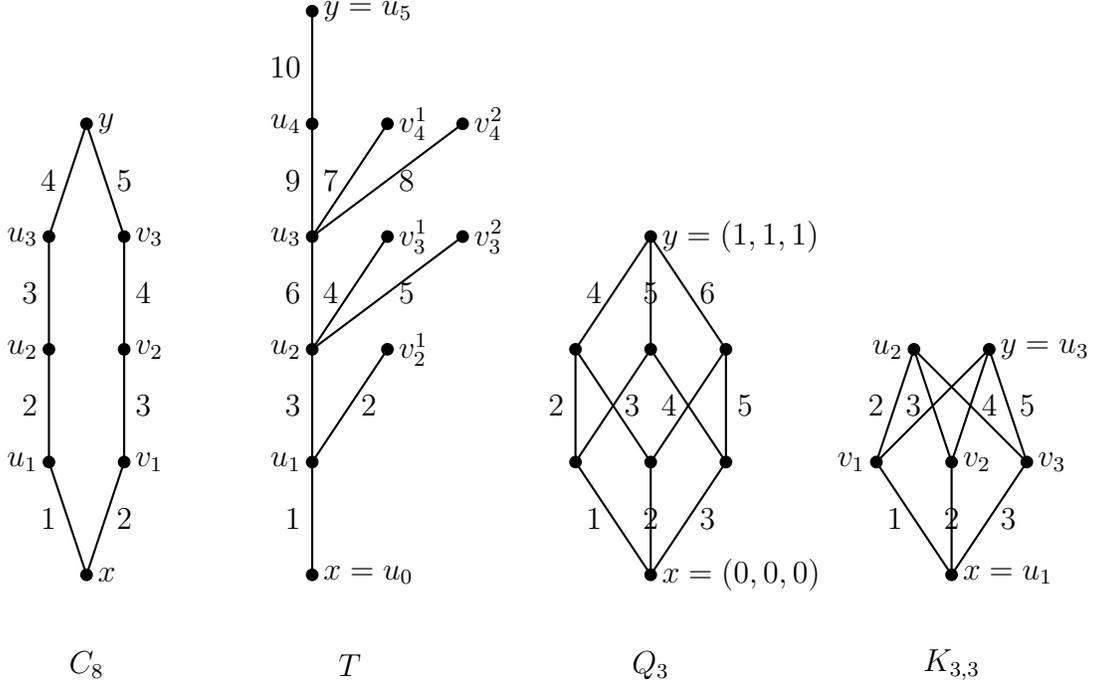
\begin{figure}[t!]
\centering
\begin{tikzpicture}[style=thick]
    \coordinate (c0) at (-0.5cm,1.5cm);
    \coordinate (c11) at (-1cm,3cm);
    \coordinate (c12) at (-1cm,4.5cm);
    \coordinate (c13) at (-1cm,6cm);
    \coordinate (c21) at (0cm,3cm);
    \coordinate (c22) at (0cm,4.5cm);
    \coordinate (c23) at (0cm,6cm);
    \coordinate (c4) at (-0.5cm,7.5cm);
    
    \draw (c0) -- node[left] {$1$} (c11) 
    -- node[left] {$2$} (c12) 
    -- node[left] {$3$} (c13) 
    -- node[left] {$4$} (c4);
    \draw (c0) -- node[right] {$2$} (c21) 
    -- node[right] {$3$} (c22) 
    -- node[right] {$4$} (c23) 
    -- node[right] {$5$} (c4);
    
    \draw[fill=black] (c0) circle (2pt)  node[right] {$x$};
    \draw[fill=black] (c11) circle (2pt) node[left] {$u_1$};
    \draw[fill=black] (c12) circle (2pt) node[left] {$u_2$};
    \draw[fill=black] (c13) circle (2pt) node[left] {$u_3$};
    \draw[fill=black] (c21) circle (2pt) node[right] {$v_1$};
    \draw[fill=black] (c22) circle (2pt) node[right] {$v_2$};
    \draw[fill=black] (c23) circle (2pt) node[right] {$v_3$};
    \draw[fill=black] (c4) circle (2pt) node[right] {$y$};
    
    \node at (-0.5cm, 0.3cm) {$C_8$};
    
    \coordinate (t0) at (2.5cm,1.5cm);
    \coordinate (t1) at (2.5cm,3cm);
    \coordinate (t21) at (2.5cm,4.5cm);
    \coordinate (t22) at (3.5cm,4.5cm);
    
    \coordinate (t31) at (2.5cm,6cm);
    \coordinate (t32) at (3.5cm,6cm);
    \coordinate (t33) at (4.5cm,6cm);
    
    \coordinate (t41) at (2.5cm,7.5cm);
    \coordinate (t42) at (3.5cm,7.5cm);
    \coordinate (t43) at (4.5cm,7.5cm);
    
    \coordinate (t5) at (2.5cm,9cm);
    
    \draw (t0) -- node[left] {$1$} (t1);
    \draw (t1) -- node[left] {$3$} (t21);
    \draw (t1) -- node[right] {$2$} (t22);
    \draw (t21) -- node[left] {$6$} (t31);
    \draw (t21) -- node[left] {$4$} (t32);
    \draw (t21) -- node[right] {$5$} (t33);
    \draw (t31) -- node[left] {$9$} (t41);
    \draw (t31) -- node[left] {$7$} (t42);
    \draw (t31) -- node[right] {$8$} (t43);
    \draw (t41) -- node[left] {$10$} (t5);
    
    \draw[fill=black] (t0) circle (2pt)  node[right] {$x=u_0$};
    \draw[fill=black] (t1) circle (2pt) node[left] {$u_1$};
    \draw[fill=black] (t21) circle (2pt) node[left] {$u_2$};
    \draw[fill=black] (t22) circle (2pt) node[right] {$v_2^1$};
    \draw[fill=black] (t31) circle (2pt) node[left] {$u_3$};
    \draw[fill=black] (t32) circle (2pt) node[right] {$v_3^1$};
    \draw[fill=black] (t33) circle (2pt) node[right] {$v_3^2$};
    \draw[fill=black] (t41) circle (2pt) node[left] {$u_4$};
    \draw[fill=black] (t42) circle (2pt) node[right] {$v_4^1$};
    \draw[fill=black] (t43) circle (2pt) node[right] {$v_4^2$};
    \draw[fill=black] (t5) circle (2pt)  node[right] {$y=u_5$};
    
    \node at (3cm, 0.3cm) {$T$};

    \coordinate (q000) at (7cm,1.5cm);
    \coordinate (q100) at (6cm,3cm);
    \coordinate (q010) at (7cm,3cm);
    \coordinate (q001) at (8cm,3cm);
    
    \coordinate (q110) at (6cm,4.5cm);
    \coordinate (q101) at (7cm,4.5cm);
    \coordinate (q011) at (8cm,4.5cm);
    \coordinate (q111) at (7cm,6cm);
    
    \draw (q000) -- node[left] {$1$} (q100);
    \draw (q000) -- node {$2$} (q010);
    \draw (q000) -- node[right] {$3$} (q001);
    
    \draw (q100) -- node[left] {$2$} (q110);
    \draw (q100) --  (q101);
    \draw (q010) -- node[right] {$3$} (q110);
    \draw (q010) --  (q011);
    \draw (q001) -- node[left] {$4$} (q101);
    \draw (q001) -- node[right] {$5$} (q011);
    
    \draw (q110) -- node[left] {$4$} (q111);
    \draw (q101) -- node {$5$} (q111);
    \draw (q011) -- node[right] {$6$} (q111);
    
    \draw[fill=black] (q000) circle (2pt)  node[right] {$x=(0,0,0)$};
    \draw[fill=black] (q100) circle (2pt);  
    \draw[fill=black] (q010) circle (2pt);  
    \draw[fill=black] (q001) circle (2pt);  
    \draw[fill=black] (q110) circle (2pt);  
    \draw[fill=black] (q101) circle (2pt);  
    \draw[fill=black] (q011) circle (2pt);  
    \draw[fill=black] (q111) circle (2pt) node[right] {$y=(1,1,1)$};  
    
    \node at (7cm, 0.3cm) {$Q_3$};

    \coordinate (b0) at (11cm,1.5cm);
    \coordinate (b11) at (10cm,3cm);
    \coordinate (b12) at (11cm,3cm);
    \coordinate (b13) at (12cm,3cm);
    
    \coordinate (b21) at (10.5cm,4.5cm);
    \coordinate (b22) at (11.5cm,4.5cm);
    
    \draw (b0) -- node[left] {$1$} (b11);
    \draw (b0) -- node {$2$} (b12);
    \draw (b0) -- node[right] {$3$} (b13);
    
    \draw (b11) -- node[left] {$2$} (b21);
    \draw (b11) -- node[left] {$3$} (b22);
    \draw (b12) --  (b21);
    \draw (b12) --  (b22);
    \draw (b13) -- node[right] {$4$} (b21);
    \draw (b13) -- node[right] {$5$} (b22);
    
    \draw[fill=black] (b0) circle (2pt)  node[right] {$x=u_1$};
    \draw[fill=black] (b11) circle (2pt)  node[left] {$v_1$};  
    \draw[fill=black] (b12) circle (2pt)  node[right] {$v_2$};  
    \draw[fill=black] (b13) circle (2pt)  node[right] {$v_3$};  
    \draw[fill=black] (b21) circle (2pt)  node[left] {$u_2$};  
    \draw[fill=black] (b22) circle (2pt)  node[right] {$y=u_3$};  
    
    \node at (11cm, 0.3cm) {$K_{3,3}$};

\end{tikzpicture}
\caption{Separable interval colorings of $C_8$, caterpillar tree $T$, $Q_{3}$ and $K_{3,3}$}
\label{separableExamples}
\end{figure}

\bigskip

\textbf{(b),(d). $\mathbf{G=T}$, $T$ is a caterpillar tree}

Let 
\begin{center}
$V(T) = \left\{u_0,u_1,v_1^1,\ldots,v_1^{k_1},u_2,v_2^1,\ldots,v_2^{k_2},\ldots,u_{n-1},v_{n-1}^1,\ldots,v_{n-1}^{k_n},u_n\right\}$
\end{center}
where $n\geq 1$, $k_i\geq 0$, $i=1,\ldots,n-1$. 

Also, let
\begin{center}
$E(T) = \left\{u_iu_{i+1} : i=0,\ldots,n-1\right\} \cup \left\{u_iv_i^j : i=1,\ldots,n-1,j=1,\ldots,k_i\right\}$
\end{center}
We define an interval $|E(T)|$-coloring $\beta$ of $T$ as follows:
\begin{center}
\begin{tabular}{lll}
$\beta(u_iu_{i+1}) = \sum\limits_{s=1}^{i}{k_s}+i+1$ & where $i=0,\ldots,n-1$ \\
$\beta(u_iv_i^j) = \sum\limits_{s=1}^{i-1}{k_s}+i+j$ & where $i=1,\ldots,n-1$, $j=1,\ldots,k_i$ \\
\end{tabular}
\end{center}
Let us now consider the level representation of $T$ with respect to $u_0$. In this case lower and upper spectrums of the vertices have the following form:
\begin{center}
\begin{tabular}{lll}
$S_{u_0}^-(u_i,\beta) = \left\{\sum\limits_{s=1}^{i}{k_s}+i+1\right\}$ & where $i=1,\ldots,n$ \\
$S_{u_0}^+(u_i,\beta) = \left[\sum\limits_{s=1}^{i-1}{k_s}+i+1,\sum\limits_{s=1}^{i}{k_s}+i+1\right]$ & where $i=0,\ldots,n-1$ \\
$S_{u_0}^-(v_i^j,\beta) = \left\{\sum\limits_{s=1}^{i-1}{k_s}+i+j\right\}$ & where $i=0,\ldots,n-1$, $j=1,\ldots,k_i$ \\
$S_{u_0}^+(v_i^j,\beta) = \emptyset$, $S_{u_0}^+(u_n,\beta) = \emptyset$
\end{tabular}
\end{center}
This implies that $\beta$ is a separable interval coloring of $T$ with respect to $u_0$. Moreover, 
\begin{center}
\begin{tabular}{lll}
$1 \in S(u_0,\beta)$\\
$\sum\limits_{s=1}^{n-1}{k_s}+n=|E(T)|=W(T) \in S(u_n,\beta)$ \\
$d_T(u_0,u_n)=n=\mathrm{diam}(T)$\\
\end{tabular}
\end{center}
So, by Theorem \ref{th:separableintervalcoloring}, we obtain the result.

\bigskip

\textbf{(c). $\mathbf{G=Q_n}$}
\begin{center}
Let $V(Q_n)=\left\{(b_1,\ldots,b_n) : b_i \in \left\{0,1\right\}\right\}$
\end{center}
In \cite{Petrosyan,PetrosyanKhachatrianTananyan}, it was shown $Q_{n}$ has an interval $\frac{n(n+1)}{2}$-coloring $\gamma$ and $W(Q_{n})=\frac{n(n+1)}{2}$. 
Moreover, by the construction of the coloring $\gamma$, it follows that $\gamma$ is a separable interval coloring with respect to $(0,\ldots,0)$ and
\begin{center}
\begin{tabular}{lll}
$1 \in S((0,\ldots,0),\gamma)$\\
$\frac{n(n+1)}{2}=W(Q_n) \in S((1,\ldots,1),\gamma)$ \\
$d_{Q_n}((0,\ldots,0),(1,\ldots,1))=n=\mathrm{diam}(Q_n)$\\
\end{tabular}
\end{center}
So, by Theorem \ref{th:separableintervalcoloring}, we obtain the result.

\textbf{(e). $\mathbf{G=K_{m,n}}$}
Let
\begin{align*}
V(K_{m,n}) &= \left\{u_1,\ldots,u_m,v_1,\ldots,v_n\right\} \\
E(K_{m,n}) &= \left\{u_iv_j : i=1,\ldots,m,\ j=1,\ldots,n\right\}
\end{align*}
In \cite{Kampreprint}, it was shown that $K_{m,n}$ has an interval $(m+n-1)$-coloring $\varphi$ ($W(K_{m,n})=m+n-1$) and this coloring can be defined as follows:
\begin{center}
$\varphi(u_iv_j)=i+j-1$, where $i=1,\ldots,m,\ j=1,\ldots,n$.
\end{center}
Let us now consider the level representation of $K_{m,n}$ with respect to $u_1$. In this case lower and upper spectrums of the vertices have the following form:
\begin{center}
\begin{tabular}{lll}
$S_{u_1}^-(v_j,\varphi)=\left\{ \varphi(u_1v_j) \right\} = \left\{ j \right\}$, & where $j=1,\ldots,n$ \\
$S_{u_1}^+(v_j,\varphi)=\left\{ \varphi(v_ju_i) : i=2,\ldots,m \right\} = \left[j+1,j+m-1 \right]$, & where $j=1,\ldots,n$ \\
$S_{u_1}^+(u_i,\varphi)=\emptyset$, & where $i=2,\ldots,m$.
\end{tabular}
\end{center}
This implies that $\varphi$ is a separable interval coloring of $K_{m,n}$ with respect to $u_1$. Moreover, 
\begin{center}
\begin{tabular}{lll}
$1 \in S(u_1,\varphi)$\\
$m+n-1=W(K_{m,n}) \in S(u_m,\varphi)$ \\
$d_{K_{m,n}}(u_1,u_m)=2=\mathrm{diam}(K_{m,n})$.\\
\end{tabular}
\end{center}
So, by Theorem \ref{th:separableintervalcoloring}, we obtain the result.
\end{proof}

\bigskip

\section{A new upper bound on $W(G)$} 

In \cite{AsrKamJCTB}, Asratian and Kamalian proved that if $G$ is a connected graph and $G\in \mathfrak{N}$, then $W(G)\leq (\mathrm{diam}(G)+1)(\Delta(G)-1)+1$. They also proved that if $G$ is a connected bipartite graph and $G\in \mathfrak{N}$, then this upper bound can be improved to $W(G)\leq \mathrm{diam}(G)(\Delta(G)-1)+1$. In this section we further improve these upper bounds. 

Let $G$ be a connected graph. Also, let $e,e^{\prime}\in E(G)$ and $e=u_{1}u_{2}$, $e^{\prime}=v_{1}v_{2}$. The distance between two edges $e$ and $e^{\prime}$ in $G$, we define as follows:

\begin{center}
$d(e,e^{\prime})=\min_{1\leq i\leq 2,1\leq j\leq 2}
d\left(u_{i},v_{j}\right)$.
\end{center}

For a graph $G$, define an edge-diameter $\mathrm{diam}_{e}(G)$ of a graph $G$ as follows:
$$\mathrm{diam}_{e}(G)=\max_{e,e'\in E(G)}d(e,e').$$

We say that $G\in \mathcal{D}(p_{1},\ldots,p_{\mathrm{diam}_{e}(G)})$ if $G\in \mathfrak{N}$ and for any pair of vertices $u,v\in V(G)$ with $d(u,v)=k$ $(1\leq k\leq \mathrm{diam}_{e}(G))$, there are distinct vertices $v_{1},\ldots,v_{p_{k}}$ such that
$d(u,v_{i})=p_{k}-1$ and $v_{i}v\in E(G)$ for $i=1,\ldots,p_{k}$. Clearly, if $G$ is a connected graph and $G\in \mathfrak{N}$, then $G\in \mathcal{D}(1,\ldots,1)$.

Let $\alpha$ be an interval $t$-coloring of $G$. Define an edge
span $\mathrm{sp}_{\alpha}\left(e,e^{\prime}\right)$ of edges $e$
and $e^\prime$ $\left(e,e^{\prime}\in E(G)\right)$ in coloring
$\alpha$ as follows:
\begin{center}
$\mathrm{sp}_{\alpha}\left(e,e^{\prime}\right)=\left\vert
\alpha(e)-\alpha(e^{\prime})\right\vert$.
\end{center}

For any $k, 0\leq k\leq \mathrm{diam}_{e}(G)$, define an edge span at distance $k$ $\mathrm{sp}_{\alpha,k}$ in coloring $\alpha$ as follows:

\begin{center}
$\mathrm{sp}_{\alpha,k}=\max
\left\{\mathrm{sp}_{\alpha}\left(e,e^{\prime}\right)\colon\,e,e^{\prime}\in
E(G)~and~d(e,e^{\prime})=k\right\}$.
\end{center}

Clearly, $\mathrm{sp}_{\alpha,0}=\Delta(G)-1$.

\begin{theorem}
\label{newupperbound} If $G$ is a connected graph and $G\in \mathcal{D}(p_{1},\ldots,p_{\mathrm{diam}_{e}(G)})$, then 
$$W\left(G\right)\leq \left(1+\mathrm{diam}_{e}(G)\right)\Delta(G)-\sum_{i=1}^{\mathrm{diam}_{e}(G)}p_{i}.$$
\end{theorem}
\begin{proof}
Let $\alpha$ be an interval $W(G)$-coloring of $G$. First we
show that if $1 \leq k \leq \mathrm{diam}_{e}(G)$, then $\mathrm{sp}_{\alpha,k}\leq
\mathrm{sp}_{\alpha,k-1}+\Delta(G)-p_{k}$.

Let $e,e^{\prime}\in E(G)$ be any two edges of $G$ with
$d(e,e^{\prime})=k$. Without loss of generality, we may assume that
$\alpha(e)\geq \alpha(e^{\prime})$. Since $d(e,e^{\prime})=k$, there
exist $u$ and $v$ vertices such that $u\in e$ and $v\in e^{\prime}$
and $d(u,v)=k$. There are $v_{1},v_{2},\ldots,v_{p_{k}}$
($v_{i}\neq v_{j}$ when $i\neq j$)
vertices such that $d(u,v_{i})=p_{k}-1$ and $v_{i}v\in E(G)$ for
$i=1,\ldots,p_{k}$. Since $G$ is a connected graph with maximum degree $\Delta(G)$, we have
\begin{center}
$\min_{1\leq i\leq p_{k}} \alpha(v_{i}v)\leq \alpha(e^{\prime})+(d_{G}(v)-p_{k}-1)+1\leq \alpha(e^{\prime})+\Delta(G)-p_{k}.$~(*)
\end{center}

Let $\alpha(e^{\prime\prime})=\min_{1\leq i\leq p_{k}} \alpha(v_{i}v)$.
By (*), we obtain
\begin{center}
$\alpha(e^{\prime})\geq \alpha(e^{\prime\prime})-(\Delta(G)-p_{k})$ and
$d(e,e^{\prime\prime})=k-1$.
\end{center}

Thus,
\begin{center}
$\mathrm{sp}_{\alpha}\left(e,e^{\prime}\right)=\left\vert
\alpha(e)-\alpha(e^{\prime})\right\vert \leq \left\vert
\alpha(e)-\alpha(e^{\prime\prime})+\Delta(G)-p_{k}\right\vert\leq \left\vert \alpha(e)-\alpha(e^{\prime\prime})\right\vert+\Delta(G)-p_{k}\leq \mathrm{sp}_{\alpha,k-1}+\Delta(G)-p_{k}$.
\end{center}

Since $e$ and $e^{\prime}$ were arbitrary edges with
$d(e,e^{\prime})=k$, we obtain $\mathrm{sp}_{\alpha,k}\leq
\mathrm{sp}_{\alpha,k-1}+\Delta(G)-p_{k}$. Now by induction on $k$ with
$\mathrm{sp}_{\alpha,0}=\Delta(G)-1$, we obtain

$$\mathrm{sp}_{\alpha,\mathrm{diam}_{e}(G)}\leq \left(1+\mathrm{diam}_{e}(G)\right)\Delta(G)-\sum_{i=1}^{\mathrm{diam}_{e}(G)}p_{i}-1.$$

From this and taking into account that $d(e,e^{\prime})\leq \mathrm{diam}_{e}(G)$ for all
$e,e^{\prime}\in E(G)$, we get 

$$W(G)\leq \left(1+\mathrm{diam}_{e}(G)\right)\Delta(G)-\sum_{i=1}^{\mathrm{diam}_{e}(G)}p_{i}.$$
\end{proof}

Let us consider connected bipartite graphs. If $G$ is a connected bipartite graph and $G\in \mathfrak{N}$, then let $uv$ and $u'v'$ be edges colored by colors $1$ and $W(G)$, respectively, in an interval $W(G)$-coloring of $G$. Since the shortest path $P$ between the vertex sets $\{u, v\}$ and $\{u',v'\}$ in $G$ has length at most $\mathrm{diam}(G)-1$, by Theorem \ref{newupperbound}, we obtain the following result.

\begin{corollary}
\label{newupperboundbip} If $G$ is a connected bipartite graph and $G\in \mathcal{D}(p_{1},\ldots,p_{\mathrm{diam}(G)-1})$, then 
$$W\left(G\right)\leq \mathrm{diam}(G)\Delta(G)-\sum_{i=1}^{\mathrm{diam}(G)-1}p_{i}.$$
\end{corollary}

Since the $k$-dimensional torus $T(2n_{1},2n_{2},\ldots,2n_{k})$ ($n_{1}\geq n_{2}\geq \cdots\geq n_{k}\geq 2$) is a connected $2k$-regular bipartite graph, $T(2n_{1},2n_{2},\ldots,2n_{k})\in \mathfrak{N}$ \cite{Petrosyanproducts} and $T(2n_{1},2n_{2},\ldots,2n_{k})\in \mathcal{D}(\underbrace{1,1,\ldots,1}_{n_{1}},\underbrace{2,2,\ldots,2}_{n_{2}},\ldots,\underbrace{k,k,\ldots,k}_{n_{k}-1})$ with $\mathrm{diam}(T(2n_{1},2n_{2},\ldots,2n_{k}))=\sum_{i=1}^{k}n_{i}$, we obtain the following upper bound on $W(T(2n_{1},2n_{2},\ldots,2n_{k}))$.

\begin{corollary}
\label{newupperboundtori}
If $n_{1},n_{2},\ldots,n_{k}\in\mathbb{N}$ and $ n_{1}\geq n_{2}\geq \cdots\geq n_{k}\geq 2$, then 
$$W(T(2n_{1},2n_{2},\ldots,2n_{k}))\leq k+\sum_{i=1}^{k}n_{i}(2k-i).$$
\end{corollary}

Let us now consider interval edge-colorings of Hamming graphs. In \cite{Petrosyanproducts}, Petrosyan proved that Hamming graphs $H(n_{1},n_{2},\ldots,n_{k})\in \mathfrak{N}$ if and only if $n_{1}\cdot n_{2}\cdots n_{k}$ is even. Moreover, if $n_{1}\cdot n_{2}\cdots n_{k}$ is even, then $w(H(n_{1},n_{2},\ldots,n_{k}))=\sum_{i=1}^{k}(n_{i}-1)$. Here, by Theorem \ref{newupperbound}, we obtain the following upper bound on $W(H(2n_{1},2n_{2},\ldots,2n_{k}))$.

\begin{theorem}\label{Hammingupper}
If $n_{1},n_{2},\ldots,n_{k}\in\mathbb{N}$, then 
$$W(H(2n_{1},2n_{2},\ldots,2n_{k}))\leq \frac{1}{2}(k+1)\sum_{i=1}^{k}(4n_{i}-3).$$
\end{theorem}
\begin{proof}
Let us first note that $\Delta(H(2n_{1},2n_{2},\ldots,2n_{k}))=\sum_{i=1}^{k}(2n_{i}-1)$ and for any two edges $e$ and $e'$ in $H(2n_{1},2n_{2},\ldots,2n_{k})$, $d(e,e')\leq k$. Since for any two vertices $u$ and $v$ at distance $r$ ($1\leq r\leq k$) in $H(2n_{1},2n_{2},\ldots,2n_{k})$, the interval $I(u,v)$ induces a hypercube of dimension $r$ in $H(2n_{1},2n_{2},\ldots,2n_{k})$ \cite{HammackImrichKlavzar}, we have that $H(2n_{1},2n_{2},\ldots,2n_{k})\in \mathcal{D}(1,2,\ldots,k)$. Now, by Theorem \ref{newupperbound}, we obtain
$$W(H(2n_{1},2n_{2},\ldots,2n_{k}))\leq (k+1)\sum_{i=1}^{k}(2n_{i}-1)-\frac{k(k+1)}{2}=(k+1)\left(\sum_{i=1}^{k}(2n_{i}-1)-\frac{k}{2}\right),$$

thus,

$$W(H(2n_{1},2n_{2},\ldots,2n_{k}))\leq \frac{1}{2}(k+1)\sum_{i=1}^{k}(4n_{i}-3).$$
\end{proof}

\begin{corollary}\label{Hamming}
If $n,k\in\mathbb{N}$, then 
$$W(H_{2n}^{k})\leq \frac{1}{2}k(k+1)(4n-3).$$
\end{corollary}

Since hypercubes $Q_{n}$ are isomorphic to Hamming graphs $H_{2}^{n}$, we
derive the following result.

\begin{corollary}\label{hypercubes}
If $n\in\mathbb{N}$, then $W\left(Q_{n}\right)
\leq \frac{n\left( n+1\right)}{2}$.
\end{corollary}
\bigskip

\section{Interval edge-colorings of $k$-dimensional tori and Hamming graphs}

In this section we begin our considerations with interval edge-colorings of $k$-dimensional tori. In \cite{Petrosyanproducts}, Petrosyan proved that $k$-dimensional tori $T(n_{1},\ldots,n_{k})\in \mathfrak{N}$ ($n_{i}\in
\mathbb{N},n_{i}\geq 3$) if and only if $n_{1}\cdot n_{2}\cdots n_{k}$ is even. Moreover, if $n_{1}\cdot n_{2}\cdots n_{k}$ is even, then $w(T(n_{1},\ldots,n_{k}))=2k$. Here, using Corollary \ref{manygraphs}, we are able to derive lower bounds on $W(T(n_{1},n_{2},\ldots,n_{k}))$.

\begin{theorem}\label{k-Tori}\
\begin{description}
    \item[1)] If $n_{1},n_{2},\ldots,n_{k}\in\mathbb{N}$ and $2\leq n_{1}\leq n_{2}\leq \cdots\leq n_{k}$, then 
$$W(T(2n_{1},2n_{2},\ldots,2n_{k}))\geq k+\sum_{i=1}^{k}n_{i}(2i-1).$$
    \item[2)] If $m_{1},m_{2},\ldots,m_{k},n_{1},n_{2},\ldots,n_{k+s}\in\mathbb{N}$ ($s\geq 0$) and $1\leq m_{1}\leq m_{2}\leq \cdots\leq m_{k}$, $2\leq n_{1}\leq n_{2}\leq \cdots\leq n_{k+s}$, then 
$$W(T(2n_{1},\ldots,2n_{k+s},2m_{1}+1,\ldots,2m_{k}+1))\geq s+2k^{2}+2\sum_{i=1}^{k}(m_{i}+n_{i})+\sum_{j=1}^{s}n_{k+j}(4k+2j-1).$$
\end{description} 
\end{theorem}
\begin{proof}
We prove 1) of Theorem \ref{k-Tori} by induction on $k$. For the case $k=1$, the statement is trivial since $W(C_{2n_{1}})=n_{1}+1$. So, assume that the lower bound on $W(T(2n_{1},2n_{2},\ldots,2n_{k-1}))$ holds.

Clearly, $T(2n_{1},2n_{2},\ldots,2n_{k-1})$ is a $(2k-2)$-regular graph and $T(2n_{1},2n_{2},\ldots,2n_{k})=C_{2n_{k}}\square T(2n_{1},2n_{2},\ldots,2n_{k-1})$. By the induction hypothesis and using Corollary \ref{manygraphs}, we have

$$W(T(2n_{1},2n_{2},\ldots,2n_{k}))\geq n_{k}+1+\left(k-1+\sum_{i=1}^{k-1}n_{i}(2i-1)\right)+n_{k}(2k-2)=k+\sum_{i=1}^{k}n_{i}(2i-1).$$

Let us now prove 2) of Theorem \ref{k-Tori}. Since the Cartesian product of graphs is commutative and associative, we have

$$T(2n_{1},\ldots,2n_{k+s},2m_{1}+1,\ldots,2m_{k}+1)\cong C_{2n_{k+s}}\square C_{2n_{k+s-1}}\square \cdots \square C_{2n_{k+1}}\square H,$$

where $H=T_{1}\square T_{2}\square \cdots \square T_{k}$ and $T_{i}=C_{2n_{i}}\square C_{2m_{i}+1}$ ($1\leq i\leq k$). By Theorem \ref{2-Tori}, we obtain $W(T_{i})\geq 2m_{i}+2n_{i}+2$ ($1\leq i\leq k$).
Since $\Delta(T_{i})=4$ ($1\leq i\leq k$), by Corollary \ref{manygraph}, we have
$$W(H)=W(T_{1}\square T_{2}\square \cdots \square T_{k})\geq 2k+\sum_{i=1}^{k}(2m_{i}+2n_{i})+2k(k-1)=2k^{2}+2\sum_{i=1}^{k}(m_{i}+n_{i}).$$
Let $G_{s}=T(2n_{1},\ldots,2n_{k+s},2m_{1}+1,\ldots,2m_{k}+1)$. We prove the lower bound on $W(G_{s})$ by induction on $s$. For the case $s=0$, we have
$$W(G_{0})=W(H)\geq 2k^{2}+2\sum_{i=1}^{k}(m_{i}+n_{i}).$$
So, assume that the lower bound on $W(G_{s-1})$ holds. Clearly, $G_{s-1}$ is a $(4k+2s-2)$-regular graph and $G_{s}=C_{2n_{k+s}}\square G_{s-1}$. By the induction hypothesis and using Corollary \ref{manygraphs}, we obtain

\begin{eqnarray*}
W(G_{s}) &\geq & n_{k+s}+1+\left(s-1+2k^{2}+2\sum_{i=1}^{k}(m_{i}+n_{i})+\sum_{j=1}^{s-1}n_{k+j}(4k+2j-1)\right)+n_{k+s}(4k+2s-2)\\ 
&=& s+2k^{2}+2\sum_{i=1}^{k}(m_{i}+n_{i})+\sum_{j=1}^{s}n_{k+j}(4k+2j-1).\\
\end{eqnarray*}
\end{proof}

\begin{corollary}\label{equitori}
For any $n\in\mathbb{N}$ and $n\geq 2$, we have 
$$W(T(\underbrace{2n,2n,\ldots,2n}_{k}))\geq nk^{2}+k.$$
\end{corollary}

Since for $n=2$ the graph $T(\underbrace{4,4,\ldots,4}_{k})$ is isomorphic to $Q_{2k}$, by Corollary \ref{hypercubes}, we obtain that $W(T(\underbrace{4,4,\ldots,4}_{k}))=W(Q_{2k})=2k^{2}+k$. This shows that the lower bound in Theorem \ref{k-Tori} is sharp.   

Finally, we consider interval edge-colorings of Hamming graphs. As we mentioned before, Hamming graphs $H(n_{1},n_{2},\ldots,n_{k})\in \mathfrak{N}$ if and only if $n_{1}\cdot n_{2}\cdots n_{k}$ is even, and if $n_{1}\cdot n_{2}\cdots n_{k}$ is even, then $w(H(n_{1},n_{2},\ldots,n_{k}))=\sum_{i=1}^{k}(n_{i}-1)$. Here, by Corollary \ref{manygraphs}, we obtain the following lower bound on $W(H(2n_{1},2n_{2},\ldots,2n_{k}))$.

\begin{theorem}\label{Hamminglower}
If $n_{1},n_{2},\ldots,n_{k}\in\mathbb{N}$, then 
$$W(H(2n_{1},2n_{2},\ldots,2n_{k}))\geq \sum_{i=1}^{k}W(K_{2n_{i}})+\sum_{i=1}^{k-1}i(2n_{i}-1).$$
\end{theorem}

In \cite{Petrosyanproducts}, Petrosyan proved that if $n=p2^{q}$, where $p$ is odd and $q$ is non-negative integer, then $W(H_{2n}^k) \geq (4n - 2 - p - q)k$. We improve this lower bound on $W(H_{2n}^k)$ for Hamming graphs $H_{2n}^k$. From Theorems \ref{Completegraph} and \ref{Hamminglower}, we have the following result.

\begin{corollary}\label{equiHamming}
If $n=\Pi_{i=1}^{\infty}p_{i}^{\alpha_{i}}$, where $p_{i}$ is the $i$th prime number and $\alpha_{i}\in \mathbb{Z}_{\geq 0}$, then
$$W(H_{2n}^k) \geq (4n - 3 - A_n)k + \frac{1}{2}k(k-1)(2n-1),$$

where $A_n = \alpha_1 + 2\alpha_2 + 3\alpha_3 + 4\alpha_4 + 4\alpha_5 + \frac{1}{2}\sum\limits_{i=6}^{\infty}{\alpha_i(p_i+1)}$.
\end{corollary}
\bigskip
	
\section{Interval edge-colorings of Fibonacci cubes}
	
In this section we consider interval edge-colorings of Fibonacci cubes $\Gamma_{n}$. Recall that $\Gamma_{n}$ is the isometric subgraph of the hypercube induced by the binary strings that contain no two consecutive 1s. Formally, it can be define as follows: let $B=\{0,1\}$ and for $n\in \mathbb{N}$, set $\mathcal{B}^{n}=\{b_{1}b_{2}\ldots b_{n}: b_{i}\in B, 1\leq i\leq n\}$. 

For $n\in \mathbb{N}$, let
$$\mathcal{F}_{n}=\{b_{1}b_{2}\ldots b_{n}\in \mathcal{B}^{n}: b_{i}\cdot b_{i+1}=0, 1\leq i\leq n-1\}.$$

The Fibonacci cube $\Gamma_{n}$ ($n\in \mathbb{N}$) has $\mathcal{F}_{n}$ as the vertex set and two vertices are adjacent if they differ in exactly one coordinate. We also need a recursive structure of Fibonacci cubes. For that, we decompose $\mathcal{F}_{n}$ into two sets $A_{n}$ and $B_{n}$ as follows:
$$A_{n}=\{b_{1}b_{2}\ldots b_{n}\in \mathcal{F}^{n}: b_{1}=1\}\text{~and~}B_{n}=\{b_{1}b_{2}\ldots b_{n}\in \mathcal{F}^{n}: b_{1}=0\}.$$

Clearly, for any $n\geq 2$, $A_{n}$ and $B_{n}$ can be recursively defined as follows: $A_{n}=\{1\alpha: \alpha\in B_{n-1}\}$ and $B_{n}=\{0\alpha: \alpha\in A_{n-1}\cup B_{n-1}\}$. This decomposition of $\mathcal{F}_{n}$ implies that the set $A_{n}$ induces a subgraph of $\Gamma_{n}$ isomorphic
to $\Gamma_{n-2}$ and the set $B_{n}$ induces $\Gamma_{n-1}$ in $\Gamma_{n}$. Moreover, each vertex $1\alpha$ of $A_{n}$ has exactly one neighbor $0\alpha$ in $B_{n}$ \cite{Klavžar}.

The interval edge-colorings of $\Gamma_{1},\Gamma_{2},\Gamma_{3}$ and $\Gamma_{4}$ are illustrated in Fig. 5.

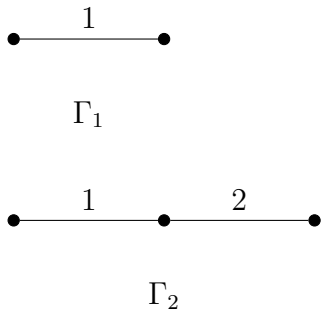
\vspace{\baselineskip}

\begin{tikzpicture}

\node[draw, circle, fill=black, inner sep=1.5pt] (v1) at (8,2) {};
\node[draw, circle, fill=black, inner sep=1.5pt] (v2) at (10,2) {};
\draw (v1) -- node[above] {1} (v2);
\node at (9, 1) {$\Gamma_1$};
\end{tikzpicture}

\vspace{\baselineskip}

\begin{tikzpicture}

\node[draw, circle, fill=black, inner sep=1.5pt] (v1) at (8,2) {};
\node[draw, circle, fill=black, inner sep=1.5pt] (v2) at (10,2) {};
\node[draw, circle, fill=black, inner sep=1.5pt] (v3) at (12,2) {};
\draw (v1) -- node[above] {1} (v2);
\draw (v2) -- node[above] {2} (v3);

\node at (10, 1) {$\Gamma_2$};

\end{tikzpicture}

\vspace{\baselineskip}

\begin{tikzpicture}

\node[draw, circle, fill=black, inner sep=1.5pt] (v1) at (8,2) {};
\node[draw, circle, fill=black, inner sep=1.5pt] (v2) at (10,2) {};
\draw (v1) -- node[above] {2} (v2);
\node[draw, circle, fill=black, inner sep=1.5pt] (v3) at (8,0) {};
\node[draw, circle, fill=black, inner sep=1.5pt] (v4) at (10,0) {};
\node[draw, circle, fill=black, inner sep=1.5pt] (v5) at (12,0) {};
\draw (v3) -- node[above] {2} (v4) -- node[above] {3} (v5);
\draw (v1) -- node[right] {1} (v3);
\draw (v2) -- node[right] {1} (v4);
\node at (10, -0.7) {$\Gamma_3$};

\end{tikzpicture}

\vspace{\baselineskip}

\begin{tikzpicture}
\node[draw, circle, fill=black, inner sep=1.5pt] (v1) at (8,4) {};
\node[draw, circle, fill=black, inner sep=1.5pt] (v2) at (10,4) {};
\draw (v1) -- node[above] {2} (v2);
\node[draw, circle, fill=black, inner sep=1.5pt] (v3) at (8,2) {};
\node[draw, circle, fill=black, inner sep=1.5pt] (v4) at (10,2) {};
\node[draw, circle, fill=black, inner sep=1.5pt] (v5) at (12,2) {};
\draw (v3) -- node[above] {2} (v4) -- node[above] {3} (v5);
\draw (v1) -- node[right] {1} (v3);
\draw (v2) -- node[right] {1} (v4);
\node[draw, circle, fill=black, inner sep=1.5pt] (v6) at (8,0) {};
\node[draw, circle, fill=black, inner sep=1.5pt] (v7) at (10,0) {};
\node[draw, circle, fill=black, inner sep=1.5pt] (v8) at (12,0) {};
\draw (v6) -- node[above] {2} (v7) -- node[above] {3} (v8);
\draw (v3) -- node[right] {3} (v6);
\draw (v4) -- node[right] {4} (v7);
\draw (v5) -- node[right] {4} (v8);
\node at (10, -0.7) {$\Gamma_4$};

\end{tikzpicture}

\begin{center}
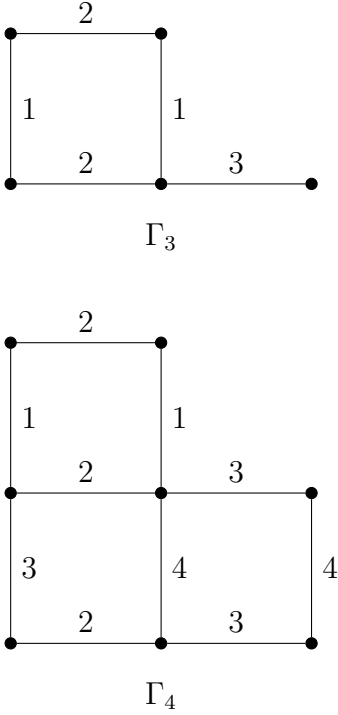
    \textit{Fig. 5. Interval $n$-colorings of $\Gamma_n$ for $n\leq 4$.}
\end{center}

\vspace{\baselineskip}
	
\begin{theorem}
\label{Fibonaccicubes} For any $n\in \mathbb{N}$, $\Gamma_{n}\in \mathfrak{N}$ and $w(\Gamma_{n})=n$. Moreover, $w(\Gamma_{n})=W(\Gamma_{n})=n$ for $n=1,2$ and if $n\geq 3$, then $W(\Gamma_{n})\geq n+1$.
\end{theorem}
\begin{proof}
For the proof of the first part of Theorem \ref{Fibonaccicubes}, we need to construct an interval $n$-coloring of $\Gamma_{n}$. It can be easily seen that Fig. 5 provides the interval $1$-coloring of $\Gamma_{1}$ and interval $2$-coloring of $\Gamma_{2}$. Let $\varphi_{1}$ and $\varphi_{2}$ be these colorings, respectively. 
As it was mentioned before Fibonacci cube $\Gamma_n$ can be decompose  into two subgraphs $\Gamma_{n - 1}$ and $\Gamma_{n - 2}$ in such a way that $V(\Gamma_n) = V(\Gamma_{n - 2}) \cup V(\Gamma_{n - 1})$ and $E(\Gamma_n) = E(\Gamma_{n - 2}) \cup E(\Gamma_{n - 1}) \cup M$, where $M$ is a matching between $A_{n}$ and $B_{n}$ \cite{Klavžar}. 

Let us now construct the edge-coloring $\varphi_{n}$ for $n \geq 3$ assuming that we have already constructed all $\varphi_k$ for $1 \leq k < n$. Here we consider two cases.

\textbf{Case 1:} $n$ is odd.

Let us color the edges of the matching between each vertex $1\alpha$ of $A_{n}$ and exactly one neighbor $0\alpha$ in $B_{n}$ with the color $\underline S\left(1\alpha,\varphi_{n-2}\right)$. For the remaining edges let us use the corresponding colors in the colorings $\varphi_{n - 2}$ and $\varphi_{n - 1}$, and color the edge $e$ with $\varphi_{n - 2}(e) + 1$ if $e \in E(\Gamma_{n - 2})$ and $\varphi_{n - 1}(e) + 1$ if $e \in E(\Gamma_{n - 1})$.

\textbf{Case 2:} $n$ is even.

We first use the corresponding colors in the colorings $\varphi_{n - 2}$ and $\varphi_{n - 1}$, and color the edge $e$ with $\varphi_{n - 2}(e) + 1$ if $e \in E(\Gamma_{n - 2})$ and $\varphi_{n - 1}(e) $ if $e \in E(\Gamma_{n - 1})$. Then we color the edges of the matching between each vertex $1\alpha$ of $A_{n}$ and exactly one neighbor $0\alpha$ in $B_{n}$ with the color $\overline S\left(1\alpha,\varphi_{n-2}\right)+2$. 

By the definition of the coloring $\varphi_{n}$, we have that:

\begin{description}
    \item[a)] if $n$ is odd, then $\underline S\left(1\alpha,\varphi_{n-2}\right)=\underline S\left(0\alpha,\varphi_{n-1}\right)$ for any $\alpha\in B_{n-1}$; 
    \item[b)] if $n$ is even, then $\overline S\left(1\alpha,\varphi_{n-2}\right)+1=\overline S\left(0\alpha,\varphi_{n-1}\right)$ for any $\alpha\in B_{n-1}$.
\end{description} 

This implies that $\varphi_{n}$ is an interval $n$-coloring of $\Gamma_{n}$; thus, $\Gamma_{n}\in \mathfrak{N}$ and $w(\Gamma_{n})\leq n$. On the other hand, since $w(\Gamma_{n})\geq \Delta(\Gamma_{n})=n$, we get $w(\Gamma_{n})=n$. 

\vspace{\baselineskip}

\begin{tikzpicture}

\node[draw, circle, fill=black, inner sep=1.5pt] (v1) at (8,2) {};
\node[draw, circle, fill=black, inner sep=1.5pt] (v2) at (10,2) {};
\draw (v1) -- node[above] {1} (v2);
\node[draw, circle, fill=black, inner sep=1.5pt] (v3) at (8,0) {};
\node[draw, circle, fill=black, inner sep=1.5pt] (v4) at (10,0) {};
\node[draw, circle, fill=black, inner sep=1.5pt] (v5) at (12,0) {};
\draw (v3) -- node[above] {3} (v4) -- node[above] {4} (v5);
\draw (v1) -- node[right] {2} (v3);
\draw (v2) -- node[right] {2} (v4);
\node at (10, -0.7) {$\Gamma_3$};

\end{tikzpicture}

\vspace{\baselineskip}

\begin{tikzpicture}
\node[draw, circle, fill=black, inner sep=1.5pt] (v1) at (8,4) {};
\node[draw, circle, fill=black, inner sep=1.5pt] (v2) at (10,4) {};
\draw (v1) -- node[above] {1} (v2);
\node[draw, circle, fill=black, inner sep=1.5pt] (v3) at (8,2) {};
\node[draw, circle, fill=black, inner sep=1.5pt] (v4) at (10,2) {};
\node[draw, circle, fill=black, inner sep=1.5pt] (v5) at (12,2) {};
\draw (v3) -- node[above] {3} (v4) -- node[above] {4} (v5);
\draw (v1) -- node[right] {2} (v3);
\draw (v2) -- node[right] {2} (v4);
\node[draw, circle, fill=black, inner sep=1.5pt] (v6) at (8,0) {};
\node[draw, circle, fill=black, inner sep=1.5pt] (v7) at (10,0) {};
\node[draw, circle, fill=black, inner sep=1.5pt] (v8) at (12,0) {};
\draw (v6) -- node[above] {3} (v7) -- node[above] {4} (v8);
\draw (v3) -- node[right] {4} (v6);
\draw (v4) -- node[right] {5} (v7);
\draw (v5) -- node[right] {5} (v8);
\node at (10, -0.7) {$\Gamma_4$};

\end{tikzpicture}

\begin{center}
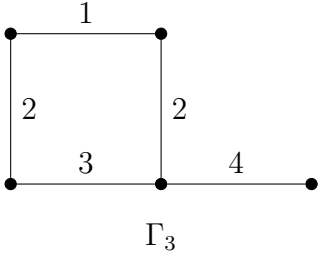
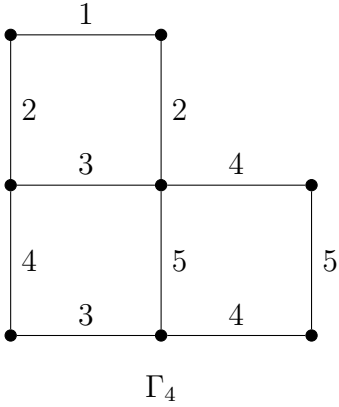
    \textit{Fig. 6. Interval $(n+1)$-colorings of $\Gamma_n$ for $n=3,4$.}
\end{center}

\vspace{\baselineskip}

Clearly, $w(\Gamma_{n})=W(\Gamma_{n})=n$ for $n=1,2$. Let us now show that if $n\geq 3$, then $\Gamma_{n}$ has an interval $(n+1)$-coloring.
It can be easily seen that Fig. 6 provides the interval $4$-coloring of $\Gamma_{3}$ and interval $5$-coloring of $\Gamma_{4}$. Let $\phi_{3}$ and $\phi_{4}$ be these colorings, respectively. Let us now construct the edge-coloring $\phi_{n}$ for $n \geq 5$ assuming that we have already constructed all $\phi_k$ for $1 \leq k < n$. Here we again consider two cases.

\textbf{Case A:} $n$ is odd.

Let us color the edges of the matching between each vertex $1\alpha$ of $A_{n}$ and exactly one neighbor $0\alpha$ in $B_{n}$ with the color $\underline S\left(1\alpha,\phi_{n-2}\right)$. For the remaining edges let us use the corresponding colors in the colorings $\phi_{n - 2}$ and $\phi_{n - 1}$, and color the edge $e$ with $\phi_{n - 2}(e) + 1$ if $e \in E(\Gamma_{n - 2})$ and $\phi_{n - 1}(e) + 1$ if $e \in E(\Gamma_{n - 1})$.

\textbf{Case B:} $n$ is even.

We first use the corresponding colors in the colorings $\phi_{n - 2}$ and $\phi_{n - 1}$, and color the edge $e$ with $\phi_{n - 2}(e) + 1$ if $e \in E(\Gamma_{n - 2})$ and $\phi_{n - 1}(e) $ if $e \in E(\Gamma_{n - 1})$. Then we color the edges of the matching between each vertex $1\alpha$ of $A_{n}$ and exactly one neighbor $0\alpha$ in $B_{n}$ with the color $\overline S\left(1\alpha,\phi_{n-2}\right)+2$. 

By the definition of the coloring $\phi_{n}$, we have that:

\begin{description}
    \item[a')] if $n$ is odd, then $\underline S\left(1\alpha,\phi_{n-2}\right)=\underline S\left(0\alpha,\phi_{n-1}\right)$ for any $\alpha\in B_{n-1}$; 
    \item[b')] if $n$ is even, then $\overline S\left(1\alpha,\phi_{n-2}\right)+1=\overline S\left(0\alpha,\phi_{n-1}\right)$ for any $\alpha\in B_{n-1}$.
\end{description} 

This implies that $\phi_{n}$ is an interval $(n+1)$-coloring of $\Gamma_{n}$ for $n\geq 3$; thus, $W(\Gamma_{n})\geq n+1$ for $n\geq 3$.
\end{proof}

\section{Concluding remarks and open problems}

In this paper we investigated interval edge-colorings of Fibonacci cubes. In the previous paper \cite{PetrosyanKhachatrianTananyan} on interval edge-colorings of Cartesian products of graphs we obtained some results on interval colorability of various meshes. Since all trees, grids, hypercubes and Fibonacci cubes are isometric subgraphs of the hypercube, it is naturally to consider the problem of interval colorability of partial cubes in general.
So, we would like to suggest the following problem.

\begin{problem}
\label{partialcubes}
Is it true that all partial cubes are interval colorable?
\end{problem}	

In \cite{Behzad}, Behzad and Mahmoodian proved that the Cartesian product of two non-trivial graphs $G$ and $H$ is planar if and only if $G\square H=P_{m}\square P_{n}$ or $G\square H=C_{m}\square C_{n}$ or $G$ is outerplanar and $H=K_{2}$. In \cite{PetrosyanKhachatrianTananyan}, we proved that if $G\square H$ is planar and both factors have at least $3$ vertices, then $G\square H\in \mathfrak{N}$ and $w(G\square H)\leq 6$, but we strongly believe that a more general result is true:    

\begin{conjecture}
\label{planarproducts}
All planar Cartesian products of two non-trivial graphs are interval colorable.
\end{conjecture}	

In fact, for the proof of Conjecture \ref{planarproducts}, it is necessary to show that all outerplanar graphs have near-interval colorings, where near-interval coloring is a proper edge-coloring of a graph such that the colors on the edges incident to any vertex is either an interval or a near-interval (it is an interval except for one missing integer). For triangle-free outerplanar graphs, it was shown by Khachatrian in \cite{Hrant}, but the problem remains open.

\begin{problem}
\label{outerplanar}
Is it true that all outerplanar graphs admit near-interval colorings?
\end{problem}	

In the previous paper \cite{PetrosyanKhachatrianTananyan} on interval edge-colorings of Cartesian products of graphs we proved that $W(Q_{n})=\frac{n\left( n+1\right)}{2}$ for the hypercubes $Q_{n}$. In this paper we were able to show that $W(\Gamma_{n})\geq n+1$ for $n\geq 3$ for the Fibonacci cubes $\Gamma_{n}$, and we strongly believe that this lower bound cannot be significantly improved.   

\begin{problem}
\label{Fibonacci}
What is the exact value of $W(\Gamma_{n})$ for any $n\geq 3$?
\end{problem}

\end{document}